\documentclass[preprint,11pt]{elsarticle_pre}
\usepackage{hyperref}

\usepackage[utf8]{inputenc}
\usepackage{float}
\usepackage{amssymb}
\usepackage[dvipsnames]{xcolor}
\usepackage{geometry}
\usepackage{marginnote}

\usepackage[english]{babel}

\usepackage{amsmath,amssymb,latexsym,amsthm,enumerate,epsfig,graphicx}%, color}%natbib,
\emergencystretch=\maxdimen
\newtheorem{theorem}{Theorem}[section]
\newtheorem{proposition}{Proposition}[section]
\newtheorem{lemma}{Lemma}[section]

\newtheorem{remark}{Remark}[section]
\newtheorem{assumption}{Assumption}[section]
\numberwithin{equation}{section}

\newif\ifcomment \commentfalse

\newcommand{\remove}[1]{}

\newenvironment{vardesc]}[1]{%
\settowidth{\parindent}{#1: \ }
\makebox{#1:}}{}

\begin{document} 

\begin{frontmatter}

\title{Convergence of a vector BGK approximation \\ for the incompressible Navier-Stokes equations}
\author[Bianchini]{Roberta Bianchini}

%\thanks{$^{\diamond}$ Dipartimento di Matematica, Universit\`a degli Studi di Roma "Tor Vergata", via della Ricerca Scientifica 1, I-00133 Rome, Italy.}
%\thanks{$^{\star}$ Istituto per le Applicazioni del Calcolo "M. Picone", Consiglio Nazionale delle Ricerche, via dei Taurini 19, I-00181 Rome, Italy.}

\ead{bianchin@mat.uniroma2.it}
%% \ead[url]{home page}
\fntext[Roberta Bianchini]{Dipartimento di Matematica, Universit\`a degli Studi di Roma "Tor Vergata", via della Ricerca Scientifica 1, I-00133 Rome, Italy - Istituto per le Applicazioni del Calcolo "M. Picone", Consiglio Nazionale delle Ricerche, via dei Taurini 19, I-00185 Rome, Italy. }
%\cortext[cor1]{}
%\address{Dipartimento di Matematica, Universit\`a degli Studi di Roma "Tor Vergata", via della Ricerca Scientifica 1, I-00133 Rome, Italy.\fnref{label3}}
 %\fntext[label3]{}

\author[Roberto Natalini]{Roberto Natalini}
\ead{roberto.natalini@cnr.it}
%\address{Istituto per le Applicazioni del Calcolo "M. Picone", Consiglio Nazionale delle Ricerche, via dei Taurini 19, I-00181 Rome, Italy.\fnref{label3}}
\fntext[Natalini]{Istituto per le Applicazioni del Calcolo "M. Picone", Consiglio Nazionale delle Ricerche, via dei Taurini 19, I-00185 Rome, Italy.}

\begin{abstract}
We present a rigorous convergence result for the smooth solutions to a singular semilinear hyperbolic approximation, called vector BGK model, to the solutions to the incompressible Navier-Stokes equations in Sobolev spaces. Our proof is based on the use of a constant right symmetrizer, weighted with respect to the parameter of the singular pertubation system. This symmetrizer provides a conservative-dissipative form for the system and this allow us to perform uniform energy estimates and to get the convergence by compactness.
\end{abstract}

\begin{keyword} 
vector BGK system \sep incompressible Navier-Stokes equations \sep symmetrizer \sep conservative-dissipative form.
\end{keyword} 

\end{frontmatter}
\section{Introduction}
We want to study the convergence of a singular perturbation approximation to the Cauchy problem for the incompressible Navier-Stokes equations on the $d$ dimensional torus $\mathbb{T}^d$:
\begin{equation}
\label{real_NS}
\begin{cases}
& \partial_{t}\textbf{u}^{NS}+\nabla \cdot (\textbf{u}^{NS} \otimes \textbf{u}^{NS}) + \nabla P^{NS}={\nu} \Delta \textbf{u}^{NS}, \\
& \nabla \cdot \textbf{u}^{NS}=0, \\
\end{cases}
\end{equation}
with $(t, x) \in [0, +\infty) \times \mathbb{T}^d,$ and initial data

\begin{equation}
\label{real_NS_initial_data}
\textbf{u}^{NS}(0,x)=\textbf{u}_{0}(x), \,\,\,\,\, \quad \,\,\,\,\, \nabla \cdot \textbf{u}_0=0.
\end{equation}

Here $\textbf{u}^{NS}$ and $\nabla P^{NS}$ are respectively the velocity field and the gradient of the pressure term, and $\nu>0$ is the viscosity coefficient. 

We consider a semilinear hyperbolic approximation, called \emph{vector BGK model}, \cite{CN, VBouchut}, to the incompressible Navier-Stokes equations (\ref{real_NS}). The general form of this approximation is as follows:

\begin{equation}
\label{BGK_generic}
\partial_{t}f_{l}^\varepsilon + \frac{\lambda_{l}}{\varepsilon} \cdot \nabla_{x}  f_{l}^\varepsilon=\frac{1}{\tau \varepsilon^{2}}(M_{l}(\rho^\varepsilon, \varepsilon \rho^\varepsilon \textbf{u}^\varepsilon)-f_{l}^\varepsilon), 
\end{equation}

with initial data 

\begin{equation}
\label{initial_conditions_BGK}
f_{l}^\varepsilon(0,x)=\bar{M}_l^\varepsilon (\bar{\rho}, \varepsilon \bar{\rho}\textbf{u}_{0}) =M_{l}^\varepsilon(\bar{\rho}, \varepsilon \bar{\rho}\textbf{u}_{0}) + \varepsilon g(\nabla\textbf{u}_0), ~~~ \textbf{u}_0 \,\,\, \text{in} \,\,\, (\ref{real_NS_initial_data}), \,\,\,\,\, l=1, \cdots, L,
\end{equation}

where $f_l^\varepsilon$ and  $M_l^\varepsilon$ 
take values in $\mathbb{R}^{d+1},$ with the Maxwellian functions $M_{l}^\varepsilon$ Lipschitz continuous, $\lambda_{l}=(\lambda_{l1}, \cdots, \lambda_{ld})$ are constant velocities, and $L \ge d+1$. The $\bar{M}_l^\varepsilon$ are the perturbed Maxwellian functions, which will be expressed later, where $g$ is the first order correction of the Maxwellians in the Chapman-Enskog expansion. Moreover,  $\bar{\rho}>0$ is a given constant value, and $\varepsilon$ and $\tau$ are positive parameters. 
Denoting by ${f_{l}}_{j}^\varepsilon, {M_{l}}_{j}^\varepsilon$, for $j=0, \cdots, d,$ the $d+1$ components of $f_{l}^\varepsilon, M_{l}^\varepsilon$ for each $l=1, \cdots, L$, let us set

\begin{equation}
\label{q_BGK}
\rho^\varepsilon=\sum_{l=1}^{L} {f_{l}}_{0}^\varepsilon(t,x) \,\,\,\,\, \text{and} \,\,\,\,\,
{q}_{j}^\varepsilon=\varepsilon \rho^\varepsilon {u}_{j}^\varepsilon=\sum_{l=1}^{L}{f_{l}}_{j}^\varepsilon(t,x).
\end{equation}

In \cite{CN, VBouchut}, the convergence of the solutions to the vector BGK model introduced above to the solutions to the incompressible Navier-Stokes equations is studied numerically. More precisely, assuming that, in a suitable functional space,
$$\displaystyle \rho^\varepsilon \rightarrow \hat{\rho}, \,\,\,\,\,\,\,\,\, \textbf{u}^\varepsilon \rightarrow \hat{\textbf{u}}, \,\,\,\,\,\,\,\,\, \text{and} \,\,\,\,\,\,\,\,\, \frac{\rho^\varepsilon - \bar{\rho}}{\varepsilon^2} \rightarrow \hat{P},$$
under some consistency conditions of the BGK approximation with respect to the Navier-Stokes equations, see \cite{CN},  
it can be shown that the couple $(\hat{\textbf{u}}, \hat{P})$ is a solution to the incompressible Navier-Stokes equations. The aim of this paper is to provide a rigorous proof of this convergence in the Sobolev spaces. 
The aim of the present paper is to provide a rigorous proof of this convergence in the Sobolev spaces. 

Vector BGK models come from the ideas of kinetic approximations for compressible flows. They are inspired by the hydrodynamic limits of the Boltzmann equation: see \cite{Golse1, Golse2, Cercignani} for the limit to the compressible Euler equations, and see \cite{Esposito, Laure} for the incompressible Navier-Stokes equations. In this regard, one of the main directions has been the approximation of hyperbolic systems with discrete velocities BGK models, as in \cite{Brenier, XinJin, Natalini, Bouchut, Perthame1}. Similar results have been obtained for convection-diffusion systems under the diffusive scaling \cite{Toscani, BGN, Lattanzio, Aregba2}.
In the framework of the BGK approximations, one of the first important contributions was given in computational physics by the so called \emph{Lattice-Boltzmann methods}, see for instance \cite{Succi, Wolf}. Under some assumptions on the physical parameters, LBMs approximate the incompressible Navier-Stokes equations by scalar velocities models of kinetic equations, and a rigorous mathematical result on the validity of these kinds of approximations was proved in \cite{Yong1}. Other partially hyperbolic approximations of the Navier-Stokes equations were developed in \cite{BNP, Raugel, Imene, Imene1}. 

The vector BGK systems studied in the present paper are a combination of the ideas of discrete velocities BGK approximations and LBMs. They are called \emph{vector BGK models} since, unlike the LBMs  \cite{Succi, Wolf}, they associate every scalar velocity with one vector of unknowns. Another fruitful property of vector BGK models is their natural compatibility with a mathematical entropy, \cite{Bouchut}, which provides a nice analytical structure and stability properties. The work of the present paper takes its roots in \cite{CN, VBouchut}, where vector BGK approximations for the incompressible Navier-Stokes equations were introduced. Here we prove a rigorous local in time convergence result for the smooth solutions to the vector BGK system to the smooth solutions to the Navier-Stokes equations. 
In this paper we focus on the two dimensional case in space. Following \cite{CN}, let us set $d=2,$ $L=5,$ and 
\begin{equation}
\label{w_BGK}
w^{\varepsilon}=(\rho^\varepsilon, \textbf{q}^\varepsilon)=(\rho^\varepsilon, q_{1}^\varepsilon, q_{2}^\varepsilon)=(\rho^\varepsilon, \varepsilon \rho^\varepsilon u_1^\varepsilon, \varepsilon\rho^\varepsilon u_2^\varepsilon)=\sum_{l=1}^{5}f_{l}^{\varepsilon} \in \mathbb{R}^3.
\end{equation}
Fix $\lambda, \tau > 0 $ and let $\varepsilon>0 $ be a small parameter, which is going to zero in the singular perturbation limit. Thus, we get a five velocities model (15 scalar equations):
\begin{equation}
\label{BGK_NS}
\begin{cases}
& \partial_{t}f_{1}^{\varepsilon} + \frac{\lambda}{\varepsilon} \partial_{x}f_{1}^{\varepsilon}=\frac{1}{\tau \varepsilon^{2}}(M_{1}(w^{\varepsilon})-f_{1}^{\varepsilon}), \\
& \partial_{t}f_{2}^{\varepsilon} + \frac{\lambda}{\varepsilon} \partial_{y}f_{2}^{\varepsilon}=\frac{1}{\tau \varepsilon^{2}}(M_{2}(w^{\varepsilon})-f_{2}^{\varepsilon}), \\
& \partial_{t}f_{3}^{\varepsilon}-\frac{\lambda}{\varepsilon}\partial_{x}f_{3}^{\varepsilon}=\frac{1}{\tau \varepsilon^{2}}(M_{3}(w^{\varepsilon})-f_{3}^{\varepsilon}), \\
& \partial_{t}f_{4}^{\varepsilon}-\frac{\lambda}{\varepsilon}\partial_{y}f_{4}^{\varepsilon}=\frac{1}{\tau \varepsilon^{2}}(M_{4}(w^{\varepsilon})-f_{4}^{\varepsilon}), \\
& \partial_{t}f_{5}^{\varepsilon}=\frac{1}{\tau \varepsilon^{2}}(M_{5}(w^{\varepsilon})-f_{5}^{\varepsilon}). \\
\end{cases}
\end{equation}
Here the Maxwellian functions $M_j \in \mathbb{R}^3$ have the following expressions:
\begin{equation}
\label{Maxwellians}
M_{1,3}(w^{\varepsilon})=aw^{\varepsilon} \pm \frac{A_{1}(w^{\varepsilon})}{2\lambda}, ~~~~~ M_{2,4}(w^\varepsilon)=aw^{\varepsilon} \pm \frac{A_{2}(w^{\varepsilon})}{2\lambda}, ~~~~~ M_{5}(w^{\varepsilon})=(1-4a)w^{\varepsilon},
\end{equation}
where
\begin{equation}
\label{Fluxes_BGK}
A_{1}(w^{\varepsilon})=\left(\begin{array}{c}
q_{1}^\varepsilon \\
\frac{(q_{1}^\varepsilon)^{2}}{\rho^\varepsilon}+P(\rho^\varepsilon)\\
\frac{q_{1}^\varepsilon q_{2}^\varepsilon}{\rho^\varepsilon}
\end{array}\right), ~~~~~
A_{2}(w^{\varepsilon})=\left(\begin{array}{c}
q_{2}^\varepsilon \\
\frac{q_{1}^\varepsilon q_{2}^\varepsilon}{\rho^\varepsilon} \\
\frac{(q_{2}^\varepsilon)^{2}}{\rho^\varepsilon}+P(\rho^\varepsilon)
\end{array}\right),
\end{equation}
\begin{equation}
\label{Pressure_approximation_BGK}
P(\rho^\varepsilon)={\rho^\varepsilon}-\bar{\rho},
\end{equation}
and
\begin{equation}
\label{a_def_BGK}
a=\frac{\nu}{2\lambda^{2}\tau},
\end{equation}
where $\nu$ is the viscosity coefficient in (\ref{real_NS}).
In the following, our main goal is to obtain uniform energy estimates for the solutions to the vector BGK model (\ref{BGK_NS}) in the Sobolev spaces and to get the convergence by compactness. In \cite{CN, VBouchut}, an $L^2$ estimate was obtained by using the entropy function associated with the vector BGK model, whose existence is proved in \cite{Bouchut}. However, there is no explicit expression for the kinetic entropy, so we do not know the weights, with respect to the singular parameter, of the terms of the classical symmetrizer derived by the entropy, see \cite{HN} for the one dimensional case and \cite{Bianchini, Yong1} for the general case. For this reason, the existence of an entropy is not enough to control the higher order estimates. Moreover, our pressure term is given by (\ref{Pressure_approximation_BGK}) and it is linear with respect to $\rho^\varepsilon,$ so the estimates in \cite{CN, VBouchut} no more hold. To solve this problem, we use a constant right symmetrizer, whose entries are weighted in terms of the singular parameter in a suitable way. Besides, the symmetrization obtained by the right multiplication provides the conservative-dissipative form introduced in \cite{Bianchini}. The dissipative property of the symmetrized system holds under the following hypothesis.
\begin{assumption} [Dissipation condition]
\label{a_condition_BGK}
We assume the following structural condition:
\begin{equation*}
0 < a < \frac{1}{4}.
\end{equation*}
\end{assumption}
Finally, we point out that Assumption \ref{a_condition_BGK} is a necessary condition, also in the case of nonlinear pressure terms, for the existence of a kinetic entropy for the approximating system, see \cite{Bouchut}.

\subsection{Plan of the paper}
In Section \ref{General_framework} we introduce the vector BGK approximation and the general setting of the problem. Section \ref{Symmetrizer_CD} is dedicated to the discussion on the symmetrizer and the conservative-dissipative form. In Section \ref{Energy_estimates} we get uniform energy estimates to prove the convergence, in Section \ref{Convergence}, of the solutions to the vector BGK approximation to the solutions to the incompressible Navier-Stokes equations. Finally, Section \ref{Conclusions} is devoted to our conclusions and perspectives.

\section{General framework}
\label{General_framework}
Let us set
\begin{equation*}
U^{\varepsilon}=(f_{1}^{\varepsilon}, f_{2}^{\varepsilon}, f_{3}^{\varepsilon}, f_{4}^{\varepsilon}, f_{5}^{\varepsilon}) \in \mathbb{R}^{3\times 5},
\end{equation*}
and let us write the compact formulation of equations (\ref{BGK_NS})-(\ref{initial_conditions_BGK}), which reads
\begin{equation}
\label{BGK_NS_compact} \partial_{t}U^{\varepsilon}+\Lambda_{1}\partial_{x}U^{\varepsilon}+\Lambda_{2}\partial_{y}U^{\varepsilon}=\frac{1}{\tau \varepsilon^{2}} (M(U^{\varepsilon})-U^{\varepsilon}), \\
\end{equation}

with initial data
\begin{equation}
\label{first_initial_data_BGK}
 U_{0}^{\varepsilon}=f_{l}^{\varepsilon}(0,x)=\bar{M}_l^\varepsilon (\bar{\rho}, \varepsilon\bar{\rho}\textbf{u}_{0})= M_{l}^\varepsilon(\bar{\rho}, \varepsilon\bar{\rho}\textbf{u}_{0})+\varepsilon g(\nabla \textbf{u}_0), \,\,\,\,\, l=1, \cdots, 5,
\end{equation}

where $\bar{M}_l^\varepsilon$ are the perturbed Maxwellian functions, with
$M_l^\varepsilon$ the Maxwellians in (\ref{Maxwellians}), and
\begin{equation}
\label{g_perturbation}
g(\nabla \textbf{u}_0)=\left(\begin{array}{c}
- a \lambda \tau  \partial_x w_0 \\
- a \lambda \tau  \partial_y w_0 \\
 a \lambda \tau  \partial_x w_0 \\
 a \lambda \tau  \partial_y w_0 \\
0
\end{array}\right), \quad w_0=(\bar{\rho}, \varepsilon \bar{\rho} \textbf{u}_0), 
\end{equation}
\begin{equation*}
\Lambda_{1}=\left(\begin{array}{ccccc}
\frac{\lambda}{\varepsilon}Id & 0 & 0 & 0 & 0 \\
0 & 0 & 0 & 0 & 0 \\
0 & 0 & -\frac{\lambda}{\varepsilon}Id & 0 & 0 \\
0 & 0 & 0 & 0 & 0 \\
0 & 0 & 0 & 0 & 0 \\
\end{array}\right), ~~~ \Lambda_{2}=\left(\begin{array}{ccccc}
0 & 0 & 0 & 0 & 0 \\
0 & \frac{\lambda}{\varepsilon}Id & 0 & 0 & 0 \\
0 & 0 & 0 & 0 & 0 \\
0 & 0 & 0 & -\frac{\lambda}{\varepsilon}Id & 0 \\
0 & 0 & 0 & 0 & 0 \\
\end{array}\right),
\end{equation*}

$Id$ is the $3 \times 3$ identity matrix, and
\begin{equation}
\label{M_BGK}
M(U^{\varepsilon})=(M_{1}^\varepsilon(w^{\varepsilon}), M_{2}^\varepsilon(w^{\varepsilon}), M_{3}^\varepsilon(w^{\varepsilon}), M_{4}^\varepsilon(w^{\varepsilon}), M_{5}^\varepsilon(w^{\varepsilon})).
\end{equation}

\subsection{Conservative variables}
We define the following change of variables:
\begin{equation}
\label{variables_BGK}
\begin{array}{l}
w^{\varepsilon}=\sum_{l=1}^{5}f_{l}^{\varepsilon}, ~~
m^{\varepsilon}=\frac{\lambda}{\varepsilon}(f_{1}^{\varepsilon}-f_{3}^{\varepsilon}), ~~
\xi^{\varepsilon}=\frac{\lambda}{\varepsilon}(f_{2}^{\varepsilon}-f_{4}^{\varepsilon}), ~~
k^{\varepsilon}=f_{1}^{\varepsilon}+f_{3}^{\varepsilon}, ~~
h^{\varepsilon}=f_{2}^{\varepsilon}+f_{4}^{\varepsilon}.
\end{array}
\end{equation}

This way, the vector BGK model (\ref{BGK_NS}) reads:
\begin{equation}
\label{BGK_NS_new_variables}
\begin{cases}
& \partial_{t}w^{\varepsilon}+\partial_{x}m^{\varepsilon}+\partial_{y}\xi^{\varepsilon}=0; \\
& \partial_{t}m^{\varepsilon} + \frac{\lambda^{2}}{\varepsilon^{2}}\partial_{x}k^{\varepsilon}=\frac{1}{\tau \varepsilon^{2}}(\frac{A_{1}(w^{\varepsilon})}{\varepsilon}-m^{\varepsilon}), \\
& \partial_{t}\xi^{\varepsilon}+\frac{\lambda^{2}}{\varepsilon^{2}}\partial_{y}h^{\varepsilon}=\frac{1}{\tau \varepsilon^{2}}(\frac{A_{2}(w^{\varepsilon})}{\varepsilon}-\xi^{\varepsilon}), \\
& \partial_{t}k^{\varepsilon}+\partial_{x}m^{\varepsilon}=\frac{1}{\tau \varepsilon^{2}}(2aw^{\varepsilon}-k^{\varepsilon}), \\
& \partial_{t}h^{\varepsilon}+\partial_{y}\xi^{\varepsilon}=\frac{1}{\tau \varepsilon^{2}}(2aw^{\varepsilon}-h^{\varepsilon}). \\
\end{cases}
\end{equation}

We make a slight modification of system (\ref{BGK_NS_new_variables}). Set $\displaystyle \bar{w}=(\bar{\rho}, 0, 0)$ and
\begin{equation}
\label{def_translation_BGK}
 w^{\varepsilon \, \star}:=w^{\varepsilon}-\bar{w}=(w_{1}^\varepsilon-\bar{\rho}, w_{2}^\varepsilon, w_{3}^\varepsilon) , ~~~ k^{\varepsilon \, \star}=k^{\varepsilon}-2a\bar{w}, ~~~ h^{\varepsilon \, \star}=h^{\varepsilon}-2a\bar{w}.
\end{equation}

In the following, we are going to work with the modified variables. System (\ref{BGK_NS_new_variables}) reads:

\begin{equation}
\label{BGK_NS_new_variables_translated}
\begin{cases}
& \partial_{t}w^{\varepsilon \, \star}+\partial_{x}m^{\varepsilon}+\partial_{y}\xi^{\varepsilon}=0; \\
& \partial_{t}m^{\varepsilon} + \frac{\lambda^{2}}{\varepsilon^{2}}\partial_{x}k^{\varepsilon \, \star}=\frac{1}{\tau \varepsilon^{2}}(\frac{A_{1}(w^{\varepsilon \, \star}+\bar{w})}{\varepsilon}-m^{\varepsilon}), \\
& \partial_{t}\xi^{\varepsilon}+\frac{\lambda^{2}}{\varepsilon^{2}}\partial_{y}h^{\varepsilon \, \star}=\frac{1}{\tau \varepsilon^{2}}(\frac{A_{2}(w^{\varepsilon \, \star}+\bar{w})}{\varepsilon}-\xi^{\varepsilon}), \\
& \partial_{t}k^{\varepsilon \, \star}+\partial_{x}m^{\varepsilon}=\frac{1}{\tau \varepsilon^{2}}(2aw^{\varepsilon \, \star}-k^{\varepsilon \, \star}), \\
& \partial_{t}h^{\varepsilon \, \star}+\partial_{y}\xi^{\varepsilon}=\frac{1}{\tau \varepsilon^{2}}(2aw^{\varepsilon \, \star} - h^{\varepsilon \, \star}). \\
\end{cases}
\end{equation}

Notice from (\ref{Fluxes_BGK}) that
$$A_{1}(w^\varepsilon)=\left(\begin{array}{c}
q_{1}^\varepsilon \\
\frac{({q_{1}^\varepsilon})^{2}}{\rho^\varepsilon}+\rho^\varepsilon-\bar{\rho} \\
\frac{q_{1}^\varepsilon q_{2}^\varepsilon}{\rho^\varepsilon}
\end{array}\right)=\left(\begin{array}{c}
w_{2}^{\varepsilon \, \star} \\
\frac{({w_{2}^{\varepsilon \, \star}})^{2}}{w_{1}^{\varepsilon \, \star}+\bar{\rho}}+w_{1}^{\varepsilon \, \star} \\
\frac{w_{2}^{\varepsilon \, \star} w_{3}^{\varepsilon \, \star}}{w_{1}^{\varepsilon \, \star}+\bar{\rho}}
\end{array}\right)=A_{1}(w^{\varepsilon \, \star}+\bar{w}),$$
and, similarly,
$$A_{2}(w^{\star})=\left(\begin{array}{c}
q_{2}^\varepsilon \\
\frac{q_{1}^\varepsilon q_{2}^\varepsilon}{\rho^\varepsilon}\\
\frac{(q_{2}^\varepsilon)^{2}}{\rho^\varepsilon}+\rho^\varepsilon-\bar{\rho} \\
\end{array}\right)=\left(\begin{array}{c}
w_{3}^{\varepsilon \, \star} \\
\frac{w_{2}^{\varepsilon \, \star} w_{3}^{\varepsilon \, \star}}{w_{1}^{\varepsilon \, \star}+\bar{\rho}}\\
\frac{(w_{3}^{\varepsilon \, \star})^2}{w_{1}^{\varepsilon \, \star}+\bar{\rho}}+w_{1}^{\varepsilon \, \star} 
\end{array}\right)=A_{2}(w^{\varepsilon \, \star}+\bar{w}).$$

Hereafter, we will omit the apexes $\varepsilon \, \star$ for $w^{\varepsilon \, \star}, k^{\varepsilon \, \star}, h^{\varepsilon \, \star}$, and the apex $\varepsilon$ for $m^\varepsilon, \xi^\varepsilon,$ when there is no ambiguity. 

Let us define the $15 \times 15$  matrix
\begin{equation}
\label{matrix_C_NS}
C=\left(\begin{array}{ccccc}
Id & Id & Id & Id & Id \\
\varepsilon \lambda Id & 0 & -\varepsilon \lambda Id & 0 & 0 \\
0 & \varepsilon \lambda Id & 0 & -\varepsilon \lambda Id & 0 \\
\varepsilon^{2}Id & 0 & \varepsilon^{2}Id & 0 & 0 \\
0 & \varepsilon^{2}Id & 0 & \varepsilon^{2}Id & 0 \\
\end{array}\right),
\end{equation}

and set
\begin{equation}
\label{W_variables_NS}
W=(w, \varepsilon^{2} m,  \varepsilon^{2} \xi,  \varepsilon^{2} k,  \varepsilon^{2} h):=CU-(\bar{w}, 0, 0, 0, 0).
\end{equation}

Thus, we can write the translated system (\ref{BGK_NS_new_variables_translated}) in the compact form
\begin{equation}
\label{BGK_NS_compact_new_variables} \partial_{t}W+B_1\partial_{x}W+B_{2}\partial_{y}W=\frac{1}{\tau \varepsilon^{2}}(\tilde{M}(W)-W), 
\end{equation}
with initial conditions
\begin{equation}
\label{initial_conditions_compact_new_variables_NS}
W_{0}=CU_{0}-(\bar{w}, 0, 0, 0, 0),
\end{equation}
where $U_0$ is given by (\ref{first_initial_data_BGK}),
$$B_1=C \Lambda_1 C^{-1}, \qquad B_2=C \Lambda_2 C^{-1},$$

\begin{equation}
\label{L1_L2_new_variables_NS}
B_{1}=\left(\begin{array}{ccccc}
0 & \frac{1}{\varepsilon^2}Id & 0 & 0 & 0 \\
0 & 0 & 0 & \frac{\lambda^2}{\varepsilon^2} & 0 \\
0 & 0 & 0 & 0 & 0 \\
0 & Id & 0 & 0 & 0 \\
0 & 0 & 0 & 0 & 0 \\
\end{array}\right), \quad
B_{2}=\left(\begin{array}{ccccc}
 0 & 0 & \frac{1}{\varepsilon^2}Id & 0 & 0 \\
 0 & 0 & 0 & 0 & 0 \\
 0 & 0 & 0 & 0 & \frac{\lambda^2}{\varepsilon^2} \\
 0 & 0 & 0 & 0 & 0 \\
 0 & 0 & Id & 0 & 0 \\
 \end{array}\right), 
\end{equation}
and 
\begin{equation*}
\tilde{M}(W)=CM(C^{-1}W)= CM(U).
\end{equation*}

Here,
\begin{equation*}
\frac{1}{\tau \varepsilon^{2}}(\tilde{M}(W)-W)=\frac{1}{\tau}\left(\begin{array}{c}
0 \\
\frac{A_{1}(w+\bar{w})}{\varepsilon}-\frac{\varepsilon^2 m}{\varepsilon^2} \\ \\
\frac{A_{2}(w+\bar{w})}{\varepsilon}-\frac{\varepsilon^2 \xi}{\varepsilon^2} \\ \\
2aw-\frac{\varepsilon^2 k}{\varepsilon^2} \\ \\
2aw-\frac{\varepsilon^2 h}{\varepsilon^2} \\ \\
\end{array}\right)=\frac{1}{\tau}\left(\begin{array}{c}
0 \\
\frac{1}{\varepsilon}\left(\begin{array}{c}
w_{2} \\
\frac{w_{2}^{2}}{w_{1}+\bar{\rho}}+w_{1}\\
\frac{w_{2} w_{3}}{w_{1}+\bar{\rho}}
\end{array}\right)-\frac{\varepsilon^2 m}{\varepsilon^2} \\
\frac{1}{\varepsilon}\left(\begin{array}{c}
w_{3} \\
\frac{w_{2}w_{3}}{w_{1}+\bar{\rho}}\\
\frac{w_{3}^2}{w_{1}+\bar{\rho}}+w_{1}
\end{array}\right)-\frac{\varepsilon^2 \xi}{\varepsilon^2} \\
2aw-\frac{\varepsilon^2 k}{\varepsilon^2} \\
2aw-\frac{\varepsilon^2 h}{\varepsilon^2} \\
\end{array}\right)
\end{equation*}

\begin{equation*}
=\frac{1}{\tau}\left(\begin{array}{ccccc}
0 & 0 & 0 & 0 & 0 \\
\frac{1}{\varepsilon}\left(\begin{array}{ccc}
0 & 1 & 0 \\
1 & 0 & 0 \\
0 & 0 & 0
\end{array}\right) & -\frac{1}{\varepsilon^2}Id & 0 & 0 & 0 \\
\frac{1}{\varepsilon}\left(\begin{array}{ccc}
0 & 0 & 1 \\
0 & 0 & 0 \\
1 & 0 & 0
\end{array}\right) & 0 & -\frac{1}{\varepsilon^2}Id & 0 & 0 \\
2aId & 0 & 0 & -\frac{1}{\varepsilon^2}Id & 0 \\
2aId & 0 & 0 & 0 & -\frac{1}{\varepsilon^2}Id \\
\end{array}\right)W
+\frac{1}{\tau}\left(\begin{array}{c}
0 \\ \\
\frac{1}{\varepsilon}\left(\begin{array}{c}
0 \\
\frac{w_{2}^{2}}{w_{1}+\bar{\rho}} \\ \\
\frac{w_{2} w_{3}}{w_{1}+\bar{\rho}}
\end{array}\right) \\ \\
\frac{1}{\varepsilon}\left(\begin{array}{c}
0 \\
\frac{w_{2} w_{3}}{w_{1}+\bar{\rho}}\\ \\
\frac{w_{3}^2}{w_{1}+\bar{\rho}}
\end{array}\right) \\ \\
0 \\
0\\
\end{array}\right)
\end{equation*}
\begin{equation}
\label{source_term_NS}
=:-LW+N(w+\bar{w}),
\end{equation}

where $-L$ is the linear part of the source term of (\ref{BGK_NS_compact_new_variables}), while $N$ is the remaining nonlinear one. Thus, we can rewrite system (\ref{BGK_NS_compact_new_variables}) as follows:
\begin{equation}
\label{BGK_NS_compact_new_variables_2} \partial_{t}W+B_{1}\partial_{x}W+B_{2}\partial_{y}W=-LW+N(w+\bar{w}).
\end{equation}

\section{The weighted constant right symmetrizer and the conservative-dissipative form}
\label{Symmetrizer_CD}
According to the theory of semilinear hyperbolic systems, see for instance \cite{Majda, Benzoni}, we need a symmetric formulation of system (\ref{BGK_NS_compact_new_variables_2}) in order to get energy estimates. However, we are dealing with a singular perturbation system, so any symmetrizer for system (\ref{BGK_NS_compact_new_variables_2}) is not enough. In other words, we look for a symmetrizer which provides a suitable dissipative structure for system (\ref{BGK_NS_compact_new_variables_2}). In this context, notice that the first equation of system (\ref{BGK_NS_compact_new_variables_2}) reads $$\partial_{t}w+\partial_{x}m+\partial_{y}\xi=0,$$
i.e. the first term of the source vanishes, and $w$ is a conservative variable. We want to take advantage of this conservative property, in order to simplify the algebraic structure of the linear part of the source term. To this end, rather than a classical Friedrichs left symmetrizer, see again \cite{Majda, Benzoni},  we look for a right symmetrizer for (\ref{BGK_NS_compact_new_variables_2}),  which allows to get the conservative-dissipative form introduced in \cite{Bianchini}. More precisely, the right multiplication easily provides the conservative structure in \cite{Bianchini}, while the dissipation is proved a posteriori.
Besides, the symmetrizer $\Sigma$ presents constant $\varepsilon$-weighted entries and this allows us to control the nonlinear part $N$ of the source term (\ref{source_term_NS}) of system (\ref{BGK_NS_compact_new_variables_2}). To be complete, we point out that the inverse matrix $\Sigma^{-1}$ is a left symmetrizer for system 
(\ref{BGK_NS_compact_new_variables_2}), according to the definitions given in \cite{Majda, Benzoni}. However, the product $-\Sigma^{-1}L$ is a full matrix, so the symmetrized version of system (\ref{BGK_NS_compact_new_variables_2}), obtained by the left multiplication by $\Sigma^{-1},$ does not provide the conservative-dissipative form in \cite{Bianchini}. 

Let us explicitly write the symmetrizer
\begin{equation}
\label{Sigma}
\Sigma=\left(\begin{array}{ccccc}
Id &\varepsilon \sigma_1 &  \varepsilon \sigma_2 & 2a\varepsilon^2 Id & 2a \varepsilon^2 Id \\
\varepsilon \sigma_1 & 2\lambda^2 a \varepsilon^2 Id & 0 &  \varepsilon^3 \sigma_1 & 0 \\
\varepsilon \sigma_2 & 0 & 2\lambda^2 a \varepsilon^2 Id & 0 & \varepsilon^3 \sigma_2 \\
2a\varepsilon^2 Id & \varepsilon^3 \sigma_1 & 0 & 2a \varepsilon^4 Id & 0 \\
2a \varepsilon^2 Id & 0 & \varepsilon^3 \sigma_2 & 0 & 2a \varepsilon^4 Id \\
\end{array}\right),
\end{equation}
where
\begin{equation}
\label{Mat12}
{\sigma}_{1}=\left(\begin{array}{ccc}
0 & 1 & 0 \\
1 & 0 & 0 \\
0 & 0 & 0 \\
\end{array}\right) ~~\text{and}~~ {\sigma}_{2}=\left(\begin{array}{ccc}
0 & 0 & 1 \\
0 & 0 & 0 \\
1 & 0 & 0 \\
\end{array}\right).
\end{equation}

It is easy to check that $\Sigma$ is a constant right symmetrizer for system (\ref{BGK_NS_compact_new_variables_2}) since, taking $B_{1}, B_{2}$ and $L$ in (\ref{L1_L2_new_variables_NS}) and (\ref{source_term_NS}) respectively, 

\begin{equation*}
B_1 \Sigma= \Sigma B_{1}^T, ~~~~~ B_{2} \Sigma= \Sigma B_{2}^{T}, 
\end{equation*}

$$-L\Sigma=\left(\begin{array}{cc}
0 & \textbf{0} \\
\textbf{0}^T & -\tilde{L} \\
\end{array}\right)$$
\begin{equation}
\label{LSigma}
=\frac{1}{\tau}\left(\begin{array}{ccccc}
0 & 0 & 0 & 0 & 0 \\
0 & -2\lambda^2 a Id + \sigma_1^2 &  \sigma_1 \sigma_2 & (2a-1)\varepsilon  \sigma_1 & 2a \varepsilon \sigma_1 \\
0 &  \sigma_1 \sigma_2 & -2\lambda^2 a Id +  \sigma_2^2 & 2a \varepsilon  \sigma_2 & (2a-1)\varepsilon  \sigma_2 \\
0 & (2a-1)\varepsilon  \sigma_1 & 2a \varepsilon  \sigma_2 & 2a(2a-1)\varepsilon^2 Id & 4a^2 \varepsilon^2 Id \\
0 & 2a\varepsilon  \sigma_1 & (2a-1)\varepsilon  \sigma_2 & 4a^2 \varepsilon^2 Id & 2a(2a-1)\varepsilon^2 Id \\
\end{array}\right).
\end{equation}

Now, we define the following change of variables:
\begin{equation}
\label{change_Sigma}
W=\Sigma \tilde{W}=\Sigma (\tilde{w}, \varepsilon^2 \tilde{m}, \varepsilon^2 \tilde{\xi}, \varepsilon^2 \tilde{k}, \varepsilon^2 \tilde{h}),
\end{equation}
with $W$ in (\ref{W_variables_NS}). System (\ref{BGK_NS_compact_new_variables_2}) reads:
\begin{equation}
\label{BGK_compact_Sigma}
\Sigma \partial_{t}\tilde{W} + B_{1}\Sigma \partial_{x}\tilde{W} + B_{2}\Sigma \partial_{y}\tilde{W} = -L\Sigma \tilde{W} + N ((\Sigma \tilde{W})_{1}+\bar{w}),
\end{equation}
where $(\Sigma \tilde{W})_{1}$ is the first component of the unknown vector $\Sigma \tilde{W}$. Now, we want to show that $\Sigma$ in (\ref{Sigma}) is strictly positive definite. Thus,
\begin{align*}
(\Sigma \tilde{W}, \tilde{W})_{0} &=||\tilde{w}||_{0}^{2}+2\lambda^2 a \varepsilon^6 (||\tilde{m}||_{0}^{2} + ||\tilde{\xi}||_{0}^{2}) + 2a \varepsilon^8 (||\tilde{k}||_{0}^{2}+||\tilde{h}||_{0}^{2}) + 2( \varepsilon^3 \sigma_1 \tilde{m}, \tilde{w})_{0} \\\\
&+2 (\varepsilon^3 \sigma_2 \tilde{\xi}, \tilde{w})_{0}+4a\varepsilon^{4} ( \tilde{k}+\tilde{h}, \tilde{w})_{0}+2\varepsilon^7( \sigma_1 \tilde{k}, \tilde{m} )_{0}
+2\varepsilon^7( \sigma_2 \tilde{h}, \tilde{\xi})_{0} \\\\
&=||\tilde{w}||_{0}^{2}+2\lambda^2 a \varepsilon^6 (||\tilde{m}||_{0}^{2} + ||\tilde{\xi}||_{0}^{2}) + 2a \varepsilon^8 (||\tilde{k}||_{0}^{2}+||\tilde{h}||_{0}^{2}) + I_{1}+I_{2}+I_{3}+I_{4}+I_{5}.
\end{align*}

Now, taking two positive constants $\delta, \mu$ and by using the Cauchy inequality, we have:
\begin{equation*}
I_{1}=2\varepsilon^3 [(\tilde{m}_{2}, \tilde{w}_{1})_{0} + (\tilde{m}_{1}, \tilde{w}_{2})_{0}] \ge - \delta \varepsilon^6 || \tilde{m}_{2}||_{0}^{2} - \frac{||\tilde{w}_{1}||_{0}^{2}}{\delta} - \delta\varepsilon^6||\tilde{m}_{1}||_{0}^{2}-\frac{||\tilde{w}_{2}||_{0}^{2}}{\delta};
\end{equation*}

\begin{equation*}
\begin{array}{l}
I_{2}=2\varepsilon^3 [(\tilde{\xi}_{3}, \tilde{w}_{1})_{0} + (\tilde{\xi}_{1}, \tilde{w}_{3})_{0}] \ge -\delta \varepsilon^6 ||\tilde{\xi}_{3}||_{0}^{2}-\frac{||\tilde{w}_{1}||_{0}^{2}}{\delta}-\delta \varepsilon^6 ||\tilde{\xi}_{1}||_{0}^2 - \frac{||\tilde{w}_{3}||_{0}^2}{\delta};
\\\\
I_{3}=4a \varepsilon^4 [(\tilde{k}, \tilde{w})_{0} + (\tilde{h}, \tilde{w})_{0}] \ge - 2a \mu ||\tilde{w}||_{0}^2 - \frac{2a \varepsilon^8}{\mu}||\tilde{k}||_{0}^2 - 2a \mu ||\tilde{w}||_{0}^2 - \frac{2a \varepsilon^8}{\mu}||\tilde{h}||_{0}^2;
\\\\
I_{4}=2\varepsilon^7 [(\tilde{k}_{2}, \tilde{m}_{1})_{0}+(\tilde{k}_{1}, \tilde{m}_{2})_{0}] \ge - \frac{\varepsilon^8}{\delta}||\tilde{k}_{2}||_{0}^2 - \delta \varepsilon^6 ||\tilde{m}_{1}||_{0}^2 - \frac{\varepsilon^8}{\delta} ||\tilde{k}_{1}||_{0}^2 -\delta \varepsilon^6 ||\tilde{m}_2||_{0}^2;
\\\\
I_{5}=2\varepsilon^7 [(\tilde{h}_{3}, \tilde{\xi}_{1})_{0}+(\tilde{h}_{1}, \tilde{\xi}_{3})_{0}] \ge - \frac{\varepsilon^8}{\delta}||\tilde{h}_{3}||_{0}^2 - \delta\varepsilon^6 ||\tilde{\xi}_{1}||_{0}^2 - \frac{\varepsilon^8}{\delta} ||\tilde{h}_{1}||_{0}^2 - \delta \varepsilon^6 ||\tilde{\xi}_3||_{0}^2.
\end{array}
\end{equation*}
Thus, putting them all together,
\begin{equation}
\label{last_positive_symmetrizer_BGK}
\begin{aligned}
(\Sigma \tilde{W}, \tilde{W})_{0} & \ge ||\tilde{w}_{1}||_{0}^2 \Bigg[1-\frac{2}{\delta}-4a\mu \Bigg]+||\tilde{w}_{2}||_{0}^2 \Bigg[1 - \frac{1}{\delta} - 4a\mu \Bigg] 
+ ||\tilde{w}_3||_0^2 \Bigg[1-\frac{1}{\delta}-4a\mu \Bigg]
\\\\
& +\varepsilon^6||\tilde{m}_1^\varepsilon||_0^2 [2\lambda^2 a - 2 \delta] +\varepsilon^6||\tilde{m}_2^\varepsilon||_0^2[2\lambda^2 a - 2 \delta] +\varepsilon^6||\tilde{m}_3^\varepsilon||_0^2[2\lambda^2 a ] 
\\\\
& +\varepsilon^6||\tilde{\xi}_1^\varepsilon||_0^2 [2\lambda^2 a - 2\delta] +\varepsilon^6||\tilde{\xi}_2^\varepsilon||_0^2[2\lambda^2 a] +\varepsilon^6||\tilde{\xi}_3^\varepsilon||_0^2[2\lambda^2 a -2\delta] 
\\\\
& +\varepsilon^8||\tilde{k}_1||_0^2 \Bigg[2a-\frac{2a}{\mu}-\frac{1}{\delta}\Bigg]+\varepsilon^8||\tilde{k}_2||_0^2 \Bigg[2a-\frac{2a}{\mu}-\frac{1}{\delta}\Bigg]+\varepsilon^8||\tilde{k}_3||_0^2 \Bigg[2a-\frac{2a}{\mu}\Bigg]
\\\\
& +\varepsilon^8||\tilde{h}_1||_0^2 \Bigg[2a-\frac{2a}{\mu}-\frac{1}{\delta}\Bigg] +\varepsilon^8||\tilde{h}_2||_0^2 \Bigg[2a-\frac{2a}{\mu}\Bigg] +\varepsilon^8||\tilde{h}_3||_0^2 \Bigg[2a-\frac{2a}{\mu}-\frac{1}{\delta}\Bigg].
\end{aligned}
\end{equation}

Now, we can prove the following lemma.
\begin{lemma}
\label{Sigma_positivity_lemma}
If Assumption \ref{a_condition_BGK} is satisfied and $\lambda$ is big enough, then $\Sigma$ is strictly positive definite.
\end{lemma}
\begin{proof}
From (\ref{last_positive_symmetrizer_BGK}), we take
\begin{equation}
\label{Sigma_conditions}
\begin{cases}
1 < \mu < \frac{1}{4a}; \\
\delta > \max \{\frac{2}{1-4a\mu}, \frac{1}{2a(1-\frac{1}{\mu})} \}; \\
\lambda > \sqrt{\frac{\delta}{a}}. \\
\end{cases}
\end{equation}

Notice that  we can choose the constant velocity $\lambda$ as big as we need, therefore the third inequality is automatically verified.
\end{proof}

Now, we consider the linear part $-L\Sigma$ of the source term of (\ref{BGK_compact_Sigma}). 

Thus,
\begin{align*}
\tau(-L\Sigma \tilde{W}, \tilde{W})_0 &  = -2\lambda^2 a \varepsilon^4 (||\tilde{m}||_0^2 + ||\tilde{\xi}||_0^2) + 2a(2a-1)\varepsilon^6 (||\tilde{k}||_0^2+||\tilde{h}||_0^2)\\\\
&+\varepsilon^4||\tilde{m}_1||_0^2
+\varepsilon^4||\tilde{m}_2||_0^2
+\varepsilon^4||\tilde{\xi}_1||_0^2+\varepsilon^4||\tilde{\xi}_3||_0^2
+2\varepsilon^4 (\sigma_1 \sigma_2 \tilde{\xi}, \tilde{m})_0\\\\
&+2(2a-1)\varepsilon^5(\sigma_1 \tilde{k}, \tilde{m})_0+4a\varepsilon^5 (\sigma_1 \tilde{h}, \tilde{m})_0
+4a\varepsilon^5 (\sigma_2 \tilde{k}, \tilde{\xi})_0\\\\
&+2(2a-1)\varepsilon^5 (\sigma_2 \tilde{h}, \tilde{\xi})_0 + 8a^2 \varepsilon^6 (\tilde{h}, \tilde{k})_0
\\\\
& =(-2\lambda^2 a + 1)\varepsilon^4 (||\tilde{m}_1||_0^2+||\tilde{m}_2||_0^2+||\tilde{\xi}_1||_0^2+||\tilde{\xi}_3||_0^2)\\\\
&-2\lambda^2 a \varepsilon^4 (||\tilde{m}_3||_0^2 + ||\tilde{\xi}_2||_0^2)
\\\\
& + 2a(2a-1)\varepsilon^6 (||\tilde{k}||_0^2+||\tilde{h}||_0^2)+J_1+J_2+J_3+J_4+J_5+J_6.
\end{align*}
Now, taking a positive constant $\omega$ and by using the Cauchy inequality, we have
\begin{align*}
& J_1=2\varepsilon^4(\tilde{\xi}_3, \tilde{m}_2)_0 \le \varepsilon^4 (||\tilde{\xi}_3||_0^2+||\tilde{m}_2||_0^2);\\\\
& J_2=(4a-2)\varepsilon^5[(\tilde{k}_2, \tilde{m}_1)_0+(\tilde{k}_1, \tilde{m}_2)_0] \le (1-2a)\Bigg\{\frac{\varepsilon^6}{\omega}(||\tilde{k}_2||_0^2+||\tilde{k}_1||_0^2) + \varepsilon^4 \omega (||\tilde{m}_1||_0^2 + ||\tilde{m}_2||_0^2)\Bigg\};\\\\
&J_3=4a\varepsilon^5 [(\tilde{h}_2, \tilde{m}_1)_0+(\tilde{h}_1, \tilde{m}_2)_0] \le 2a \Bigg\{ \frac{\varepsilon^6}{\omega}||\tilde{h}_2||_0^2+\varepsilon^4 \omega||\tilde{m}_1||_0^2+\frac{\varepsilon^6}{\omega}||\tilde{h}_1||_0^2 + \varepsilon^4 \omega ||\tilde{m}_2||_0^2 \Bigg\};\\\\
&J_4=4a\varepsilon^5 [(\tilde{k}_3, \tilde{\xi}_1)_0+(\tilde{k}_1, \tilde{\xi}_3)_0] \le 2a \Bigg\{ \frac{\varepsilon^6}{\omega}||\tilde{k}_3||_0^2+\varepsilon^4 \omega ||\tilde{\xi}_1||_0^2 + \frac{\varepsilon^6}{\omega}||\tilde{k}_1||_0^2+\varepsilon^4 \omega ||\tilde{\xi}_3||_0^2 \Bigg\};\\\\
&J_5=2(2a-1)\varepsilon^5[(\tilde{h}_3, \tilde{\xi}_1)_0+(\tilde{h}_1, \tilde{\xi}_3)_0] \le (1-2a) \Bigg\{ \frac{\varepsilon^6}{\omega}||\tilde{h}_3||_0^2+\varepsilon^4 \omega ||\tilde{\xi}_1||_0^2 + \frac{\varepsilon^6}{\omega}||\tilde{h}_1||_0^2+\varepsilon^4 \omega ||\tilde{\xi}_3||_0^2 \Bigg\};\\\\
&J_6=8a^2\varepsilon^6(\tilde{h}, \tilde{k})_0 \le 4a^2 \varepsilon^6 \{{||\tilde{h}||_0^2}+||\tilde{k}||_0^2 \}.
\end{align*}

Putting them all together, we have
\begin{equation}
\label{last_negative_source}
\begin{aligned}
&\tau(-L\Sigma \tilde{W}, \tilde{W})_0 \le \varepsilon^4||\tilde{m}_1||_0^2 [-2\lambda^2 a + 1 +\omega] + \varepsilon^4 ||\tilde{m}_2||_0^2 [-2\lambda^2 a +2+\omega]
-2\lambda^2 a \varepsilon^4 ||\tilde{m}_3||_0^2\\\\
&+ \varepsilon^4 ||\tilde{\xi}_1||_0^2 [-2 \lambda^2 a + 1 + \omega]-2\lambda^2 a \varepsilon^4 ||\tilde{\xi}_2||_0^2
+ \varepsilon^4 ||\tilde{\xi}_3||_0^2 [-2\lambda^2 a + 2 + \omega] + \varepsilon^6 ||\tilde{k}_1||_0^2 \Bigg[2a(4a-1)+\frac{1}{\omega}\Bigg]\\\\
&+ \varepsilon^6 ||\tilde{k}_2||_0^2 \Bigg[2a(4a-1)+\frac{(1-2a)}{\omega}\Bigg] + \varepsilon^6 ||\tilde{k}_3||_0^2 \Bigg[2a(4a-1)+\frac{2a}{\omega}\Bigg] + \varepsilon^6 ||\tilde{h}_1||_0^2 \Bigg[2a(4a-1)+\frac{1}{\omega}\Bigg]\\\\
&+\varepsilon^6||\tilde{h}_2||_0^2\Bigg[2a(4a-1)+\frac{2a}{\omega}\Bigg]+\varepsilon^6||\tilde{h}_3||_0^2 \Bigg[2a(4a-1)+\frac{(1-2a)}{\omega}\Bigg].
\end{aligned}
\end{equation}

This way, we obtain the following Lemma.

\begin{lemma}
\label{Negative_Source}
If Assumption \ref{a_condition_BGK} is satisfied and $\lambda$ is big enough, then the symmetrized linear part of the source term $-L\Sigma$ given by (\ref{LSigma}) is negative definite.
\begin{proof}
We need $\omega$ and $\lambda$ satisfying:
\begin{equation}
\label{conditions_negativity}
\begin{cases}
\omega > \frac{1}{2a(1-4a)}; \\\\
\lambda > \sqrt{\frac{2+\omega}{2a}}.
\end{cases}
\end{equation}
Recalling (\ref{Sigma_conditions}), we assume 
\begin{equation}
\label{lambda_condition_BGK}
\lambda > \max\Bigg\{\sqrt{\frac{\delta}{a}}, \, \sqrt{\frac{4a(1-4a)+1}{4a^2(1-4a)}}\Bigg\}.
\end{equation}
Then, we take $\omega > \frac{1}{2a(1-4a)},$ which ends the proof.
\end{proof}
\end{lemma}

\section{Energy estimates}
\label{Energy_estimates}
Here we provide $\varepsilon$-weighted energy estimates for the solution $W^\varepsilon$ to (\ref{BGK_NS_compact_new_variables_2}). Let us introduce $T^\varepsilon$ the maximal time of existence of the unique solution $\tilde{W}^\varepsilon$ for fixed $\varepsilon$ to system (\ref{BGK_compact_Sigma}),  see \cite{Majda}. In the following, we consider the time interval $[0,T],$ with $T \in [0, T^\varepsilon)$. Our setting is represented by the Sobolev spaces $H^s(\mathbb{T}^2),$ with $s>3.$
\subsection{Zero order estimate}
We assume the following condition.
\begin{assumption}
\label{lambda_assumption_BGK}
Let $\lambda$ satisfies (\ref{lambda_condition_BGK}) and 
\begin{equation}
\label{lambda_condition_estimate1_BGK}
\lambda > \sqrt{\frac{5+\frac{1}{a(1-4a)}}{4a}}.
\end{equation}
\end{assumption}
\begin{lemma}
If Assumptions \ref{a_condition_BGK} and \ref{lambda_assumption_BGK} are satisfied,
then the following zero order energy estimate holds:
\begin{equation*}
||\tilde{w}(T)||_0^2 +  \varepsilon^6 (||\tilde{m}(T)||_0^2 + ||\tilde{\xi}(T)||_0^2) + \varepsilon^8 (||\tilde{k}(T)||_0^2+||\tilde{h}(T)||_0^2)
\end{equation*}
\begin{equation*}
+ \int_0^T  \varepsilon^4 (||\tilde{m}(t)||_0^2+||\tilde{\xi}(t)||_0^2) + \varepsilon^6 (||\tilde{k}(t)||_0^2+||\tilde{h}(t)||_0^2) ~ dt \le c\varepsilon^2 (||\textbf{u}_0||_0^2 + || \nabla \textbf{u}_0 ||_0^2)
\end{equation*}
\begin{equation}
\label{zero_order_estimate_lemma_BGK}
+ c(||\textbf{u}||_{L^\infty([0,T]\times\mathbb{T}^2)}) \int_0^T ||\tilde{w}(t)||_0^2 +\varepsilon^6(||\tilde{m}(t)||_0^2 + ||\tilde{\xi}(t)||_0^2) + \varepsilon^8 (||\tilde{k}(t)||_0^2+||\tilde{h}(t)||_0^2) ~ dt.
\end{equation} 
\end{lemma}

\begin{proof}
We consider the symmetrized compact system (\ref{BGK_compact_Sigma}) and we multiply $\tilde{W}$ through the $L^2$-scalar product. Thus, we have:
\begin{equation*}
\frac{1}{2}\frac{d}{dt}(\Sigma \tilde{W}, \tilde{W})_0 + (L\Sigma \tilde{W}, \tilde{W})_0 = (N((\Sigma \tilde{W})_1+\bar{w}), \tilde{W})_0.
\end{equation*}
Integrating in time, we get:
$$\frac{1}{2}(\Sigma \tilde{W}(T), \tilde{W}(T))_0 + \int_0^T (L\Sigma \tilde{W}(t), \tilde{W}(t))_0 ~ dt  \le \frac{1}{2}(\Sigma \tilde{W}(0), \tilde{W}(0))_0$$
\begin{equation}
\label{time_integral_BGK}
+ \int_0^T  |(N((\Sigma\tilde{W}(t))_1+\bar{w}), \tilde{W}(t))_0| ~ dt.
\end{equation}

Consider (\ref{last_positive_symmetrizer_BGK}) and let us introduce the following positive constants:
\begin{equation}
\label{functional_Sigma}
\begin{array}{lll}
\Gamma_{\Sigma}:=1-4a\mu -\frac{2}{\delta}, & \Delta_\Sigma:=2(\lambda^2 a - \delta) &
\Theta_\Sigma:=2a(1-\frac{1}{\mu})-\frac{1}{\delta}.
\end{array}
\end{equation}

Similarly, from (\ref{last_negative_source}), we define:
\begin{equation}
\label{functional_source_BGK}
\begin{array}{ll}
\Delta_{L\Sigma}:=2(\lambda^2a-1)-\omega, & \Theta_{L\Sigma}:=2a(1-4a)-\frac{1}{\omega}.
\end{array}
\end{equation}

Thus,  from (\ref{time_integral_BGK}), we get:
\begin{equation*}
\Gamma_\Sigma ||\tilde{w}(T)||_0^2 + \varepsilon^6\Delta_\Sigma  (||\tilde{m}(T)||_0^2 + ||\tilde{\xi}(T)||_0^2) + \varepsilon^8 \Theta_\Sigma  (||\tilde{k}(T)||_0^2+||\tilde{h}(T)||_0^2) 
\end{equation*}
\begin{equation*}
+ \frac{2}{\tau} \int_0^T  \varepsilon^4\Delta_{L\Sigma} (||\tilde{m}(t)||_0^2+||\tilde{\xi}(t)||_0^2) +  \varepsilon^6 \Theta_{L\Sigma} (||\tilde{k}(t)||_0^2+||\tilde{h}(t)||_0^2) ~ dt
\end{equation*}
\begin{equation}
\label{time_integral_2_BGK}
\le (\Sigma \tilde{W}_0, \tilde{W}_0)_0+2\int_0^T  |(N((\Sigma\tilde{W}(t))_1+\bar{w}), \tilde{W}(t))_0| ~ dt. 
\end{equation}

Notice that, from (\ref{change_Sigma}), 
\begin{equation*}
(\Sigma \tilde{W}_0, \tilde{W}_0)_0=(\Sigma \Sigma^{-1}W_0, \Sigma^{-1}W_0)_0 = (\Sigma^{-1} W_0, W_0)_0,
\end{equation*}
where $ \displaystyle W_0=W(0,x)=(w(0,x), \varepsilon^2{m}(0,x), \varepsilon^2 {\xi}(0,x), \varepsilon^2{k}(0,x), \varepsilon^2 {h}(0,x)), $ and, from (\ref{variables_BGK}), (\ref{def_translation_BGK}) and the initial conditions (\ref{initial_conditions_BGK}),
\begin{equation*}
\begin{array}{l}
w(0,x)=w_0-\bar{w}=(0, \varepsilon \bar{\rho} {u_0}_1, \varepsilon \bar{\rho} {u_0}_2); \\ 
m(0,x)=\frac{\lambda}{\varepsilon}({f_1}_0-{f_3}_0)=\frac{A_1(w_0)}{\varepsilon}-2a \lambda^2 \tau \partial_x w_0\\
\qquad \quad \,=(\bar{\rho}{u_0}_1, \varepsilon \bar{\rho} {{u_0}_1}^2-2a \varepsilon \bar{\rho} \partial_x {u_0}_1, \varepsilon \bar{\rho} {{u_0}_1}{u_0}_2-2a \varepsilon \bar{\rho} \partial_x {u_0}_2); \\
\xi(0,x)=\frac{\lambda}{\varepsilon}({f_2}_0-{f_4}_0)=\frac{A_2(w_0)}{\varepsilon}-2a \lambda^2 \tau \partial_y w_0\\
\qquad \quad \,=(\bar{\rho}{u_0}_2, \varepsilon \bar{\rho} {{u_0}_1}{u_0}_2-2a \varepsilon \bar{\rho} \partial_y {u_0}_1, \varepsilon \bar{\rho}{{u_0}_2}^2-2a \varepsilon \bar{\rho} \partial_y {u_0}_2 ); \\
k(0,x)={f_1}_0+{f_3}_0-2a\bar{w}=2aw_0-2a\bar{w}=2a(0, \varepsilon \bar{\rho} {u_0}_1, \varepsilon \bar{\rho} {u_0}_2); \\ 
h(0,x)={f_2}_0+{f_4}_0-2a\bar{w}=2aw_0-2a\bar{w}=2a(0, \varepsilon \bar{\rho} {u_0}_1, \varepsilon \bar{\rho} {u_0}_2). \\ 
\end{array}
\end{equation*}

Besides, the explicit expression of the constant symmetric matrix $\Sigma^{-1}$ is given by
\begin{equation}
\label{Inv_Sigma}
\Sigma^{-1}=\left(\begin{array}{ccccc}
\frac{1}{1-4a}Id & 0 & 0 & \frac{-1}{\varepsilon^2(1-4a)}Id & \frac{-1}{\varepsilon^2(1-4a)}Id \\
0 & H_1 & 0 & \frac{1}{\varepsilon^3(1-4\lambda^2 a^2)}\sigma_1 & 0 \\
0 & 0 & H_2 & 0 & \frac{1}{\varepsilon^3(1-4\lambda^2 a^2)}\sigma_2 \\
\frac{-1}{\varepsilon^2(1-4a)}Id &  \frac{1}{\varepsilon^3(1-4\lambda^2 a^2)}\sigma_1 & 0 & H_3 & \frac{1}{\varepsilon^4(1-4a)}Id \\
\frac{-1}{\varepsilon^2(1-4a)}Id & 0 &  \frac{1}{\varepsilon^3(1-4\lambda^2 a^2)}\sigma_2 & \frac{1}{\varepsilon^4(1-4a)}Id & H_4 \\
\end{array}\right),
\end{equation}

where
\begin{equation*}
H_1=\left(\begin{array}{ccc}
\frac{2a}{\varepsilon^2 (4\lambda^2 a^2 -1)} & 0 & 0 \\
0 & \frac{2a}{\varepsilon^2 (4\lambda^2 a^2 -1)} & 0 \\
0 & 0 & \frac{1}{2\lambda^2 a \varepsilon^2} \\
\end{array}\right); ~~~ H_2=\left(\begin{array}{ccc}
\frac{2a}{\varepsilon^2 (4\lambda^2 a^2 -1)} & 0 & 0 \\
0 & \frac{1}{2\lambda^2 a \varepsilon^2} & 0 \\
0 & 0 & \frac{2a}{\varepsilon^2 (4\lambda^2 a^2 -1)} \\
\end{array}\right);
\end{equation*}
\begin{equation*}
H_3=\left(\begin{array}{ccc}
\frac{4\lambda^2 a^2 - 2 \lambda^2 a +1}{\varepsilon^4 (4a-1)(4\lambda^2 a^2-1)} & 0 & 0 \\
0 & \frac{4\lambda^2 a^2 - 2 \lambda^2 a +1}{\varepsilon^4 (4a-1)(4\lambda^2 a^2-1)} & 0 \\
0 & 0 & \frac{2a-1}{2a\varepsilon^4(4a-1)}\\
\end{array}\right);
\end{equation*} 
\begin{equation*}
H_4=\left(\begin{array}{ccc}
\frac{4\lambda^2 a^2 - 2 \lambda^2 a +1}{\varepsilon^4 (4a-1)(4\lambda^2 a^2-1)} & 0 & 0 \\
0 & \frac{2a-1}{2a\varepsilon^4(4a-1)} & 0 \\
0 & 0 & \frac{4\lambda^2 a^2 - 2 \lambda^2 a +1}{\varepsilon^4 (4a-1)(4\lambda^2 a^2-1)} \\
\end{array}\right).
\end{equation*}

It is easy to check that
\begin{align*}
(\Sigma^{-1}W_0, W_0)_0 & = \bar{\rho}^2 \varepsilon^2 \|\textbf{u}_0\|_0^2 + \frac{2a\bar{\rho}^2\varepsilon^4}{4\lambda^2 a^2-1} (\|{u_{0}}_1^2-2a \lambda^2 \partial_x {u_0}_1\|_0^2+\|{u_{0}}_2^2-2a \lambda^2 \partial_y {u_0}_2\|_0^2) \\
& +\frac{\bar{\rho}^2\varepsilon^4}{2 \lambda^2 a}(\|{u_0}_1{u_0}_2-2a \lambda^2 \partial_x {u_0}_2\|_0^2+\|{u_0}_1{u_0}_2-2a \lambda^2 \partial_y {u_0}_1\|_0^2) \\
& \le c \varepsilon^2 (\|\textbf{u}_0\|_0^2 + \|\nabla \textbf{u}_0 \|_0^2),
\end{align*}
and so, from (\ref{time_integral_2_BGK}) we get the following inequality:
\begin{equation}
\label{time_integral_3_BGK}
\begin{array}{c}
\Gamma_\Sigma ||\tilde{w}(T)||_0^2 + \varepsilon^6\Delta_\Sigma  (||\tilde{m}(T)||_0^2 + ||\tilde{\xi}(T)||_0^2) +  \varepsilon^8\Theta_\Sigma (||\tilde{k}(T)||_0^2+||\tilde{h}(T)||_0^2) 
\\\\

+  \displaystyle \frac{2}{\tau} {\int_0^T}  \varepsilon^4\Delta_{L\Sigma}  (||\tilde{m}(t)||_0^2+||\tilde{\xi}(t)||_0^2) + \varepsilon^6\Theta_{L\Sigma}(||\tilde{k}(t)||_0^2+||\tilde{h}(t)||_0^2) ~ dt \\\\

\le c \varepsilon^2 (\|\textbf{u}_0\|_0^2 + \|\nabla \textbf{u}_0 \|_0^2) + 2\displaystyle{\int_0^T}   |(N((\Sigma\tilde{W}(t))_1+\bar{w}), \tilde{W}(t))_0| ~ dt. 
\end{array}
\end{equation}
It remains to deal with the last term of (\ref{time_integral_3_BGK}). Recall that $w=(\rho-\bar{\rho}, \varepsilon \rho {u}_1, \varepsilon \rho u_2)$. From (\ref{source_term_NS}),
\begin{equation}
\label{N_source_BGK}
N((\Sigma \tilde{W})_1+\bar{w})=N(w+\bar{w})=
\frac{1}{\tau}\left(\begin{array}{c}
0 \\ \\
\left(\begin{array}{c}
0 \\
u_1 w_2 \\
u_1 w_3 \\
\end{array}\right) \\ \\
\left(\begin{array}{c}
0 \\
u_2 w_2 \\
u_2 w_3 \\
\end{array}\right) \\ \\
0 \\
0
\end{array}\right).
\end{equation}

Thus,
\begin{equation*}
(N(w+\bar{w}), \tilde{W})_0=\frac{1}{\tau}\{(u_1 w_2, \varepsilon^2 \tilde{m}_2)_0+(u_1 w_3, \varepsilon^2 \tilde{m}_3)_0+(u_2 w_2, \varepsilon^2 \tilde{\xi}_2)_0+(u_2 w_3, \varepsilon^2 \tilde{\xi}_3)_0\}
\end{equation*}
\begin{equation*}
\le \frac{1}{2\tau}\{||u_1w_2||_0^2+\varepsilon^4 ||\tilde{m}_2||_0^2 + ||u_1w_3||_0^2 + \varepsilon^4 ||\tilde{m}_3||_0^2+|| u_2w_2||_0^2 + \varepsilon^4 ||\tilde{\xi}_2||_0^2 + ||u_2w_3||_0^2 + \varepsilon^4 ||\tilde{\xi}_3||_0^2\}
\end{equation*}
\begin{equation*}
\le c(||\textbf{u}||_\infty) ||w||_0^2 + \frac{\varepsilon^4}{2\tau}(||\tilde{m}||_0^2+||\tilde{\xi}||_0^2).
\end{equation*}

By definition (\ref{change_Sigma}), explicitly we have:
\begin{equation}
\label{w_star_w_star_tilde}
w=(\Sigma \tilde{W}^\varepsilon)_1=\tilde{w}+\varepsilon^3 \sigma_1 \tilde{m} + \varepsilon^3 \sigma_2 \tilde{\xi} + 2a\varepsilon^4 (\tilde{k}+\tilde{h}), 
\end{equation}
and so,
\begin{equation*}
|(N(w+\bar{w}), \tilde{W})_0| \le c(||\textbf{u}||_{\infty}) \{ ||\tilde{w}||_0^2  + \varepsilon^6(||\tilde{m}||_0^2 + ||\tilde{\xi}||_0^2) + \varepsilon^8 (||\tilde{k}||_0^2+||\tilde{h}||_0^2) \} +
\frac{\varepsilon^4}{2\tau} (||\tilde{m}||_0^2 + ||\tilde{\xi}||_0^2).
\end{equation*}

Putting them all together, (\ref{time_integral_3_BGK}) yields:
\begin{equation}
\label{time_integral_4}
\begin{array}{c}

\Gamma_\Sigma ||\tilde{w}(T)||_0^2 +  \varepsilon^6\Delta_\Sigma (||\tilde{m}(T)||_0^2 + ||\tilde{\xi}(T)||_0^2) + \varepsilon^8 \Theta_\Sigma  (||\tilde{k}(T)||_0^2+||\tilde{h}(T)||_0^2) \\\\

+  \displaystyle \frac{2}{\tau}{\int_0^T}   \varepsilon^4\Delta_{L\Sigma}  (||\tilde{m}(t)||_0^2+||\tilde{\xi}(t)||_0^2) + \varepsilon^6\Theta_{L\Sigma} (||\tilde{k}(t)||_0^2+||\tilde{h}(t)||_0^2) ~ dt

\\\\
\le c \varepsilon^2 (\|\textbf{u}_0\|_0^2 + \|\nabla \textbf{u}_0 \|_0^2) +  \displaystyle{\int_0^T} {\frac{\varepsilon^4}{\tau}} (||\tilde{m}(t)||_0^2 + ||\tilde{\xi}(t)||_0^2 ) ~ dt
\\\\

+ c(||\textbf{u}||_{L^\infty([0,T]\times\mathbb{T}^2)}) \displaystyle{\int_0^T} ||\tilde{w}(t)||_0^2 +\varepsilon^6(||\tilde{m}(t)||_0^2 + ||\tilde{\xi}(t)||_0^2) + \varepsilon^8 (||\tilde{k}(t)||_0^2+||\tilde{h}(t)||_0^2) ~ dt.
\end{array}
\end{equation}

This gives:
\begin{equation}
\label{zero_order_estimate_BGK}
\begin{array}{c}

\Gamma_\Sigma ||\tilde{w}(T)||_0^2 +  \varepsilon^6\Delta_\Sigma (||\tilde{m}(T)||_0^2 + ||\tilde{\xi}(T)||_0^2) + \varepsilon^8 \Theta_\Sigma  (||\tilde{k}(T)||_0^2+||\tilde{h}(T)||_0^2) 

\\\\
+  \displaystyle \frac{1}{\tau}{\int_0^T}   \varepsilon^4 (2{\Delta}_{L\Sigma}-1)  (||\tilde{m}(t)||_0^2+||\tilde{\xi}(t)||_0^2) + 2\varepsilon^6 \Theta_{L\Sigma} (||\tilde{k}(t)||_0^2+||\tilde{h}(t)||_0^2)  ~ dt 

\\\\
\le c \varepsilon^2 (\|\textbf{u}_0\|_0^2 + \|\nabla \textbf{u}_0 \|_0^2)

\\\\\
+ c(||\textbf{u}||_{L^\infty([0,T]\times\mathbb{T}^2)}) \displaystyle{\int_0^T} ||\tilde{w}(t)||_0^2 +\varepsilon^6(||\tilde{m}(t)||_0^2 + ||\tilde{\xi}(t)||_0^2) + \varepsilon^8 (||\tilde{k}(t)||_0^2+||\tilde{h}(t)||_0^2) ~ dt,
\end{array}
\end{equation} 
where, by definition (\ref{functional_source_BGK}), $2\Delta_{L\Sigma}-1=4\lambda^2 a - 4 - 2\omega$ is positive thanks to condition (\ref{lambda_condition_estimate1_BGK}). This gives estimate (\ref{zero_order_estimate_lemma_BGK}).
\end{proof}

\subsection{Higher order estimates}
\begin{lemma}
If Assumptions \ref{a_condition_BGK} and \ref{lambda_assumption_BGK} are satisfied, then the following $H^s$ energy estimate holds:
\begin{equation*}
\begin{array}{c}
|| \tilde{w}(T)||_s^2 + \varepsilon^6 (||\tilde{m}(T)||_s^2 + ||\tilde{\xi}(T)||_s^2) +  \varepsilon^8 (||\tilde{k}(T)||_s^2+||\tilde{h}(T)||_s^2) 
\\\\

+  \displaystyle{\int_0^T}  \varepsilon^4  (||\tilde{m}(t)||_s^2+||\tilde{\xi}(t)||_s^2) + \varepsilon^6 (||\tilde{k}(t)||_s^2+||\tilde{h}(t)||_s^2) ~ dt \le c\varepsilon^2 (||\textbf{u}_0||_s^2+|| \nabla \textbf{u}_0 ||_s^2)
\\\\

+ c(||\textbf{u}\,||_{L_t^\infty H_x^s}) \displaystyle{\int_0^T}||\tilde{w}(t)||_s^2+\varepsilon^6(||\tilde{m}(t)||_s^2+||\tilde{\xi}(t)||_s^2)+\varepsilon^8(||\tilde{k}(t)||_s^2+||\tilde{h}(t)||_s^2) ~ dt,
\end{array}
\end{equation*}
where, hereafter,
$$\displaystyle {L_t^\infty H_x^{s}}:=L^\infty([0,T], H^s(\mathbb{T}^2)), \quad \text{for} \,\, s \in \mathbb{R}.$$
\end{lemma}

\begin{proof}
We take the $|\alpha|$-derivative, $0 < |\alpha | \le s$, of the semilinear system given by (\ref{BGK_NS_compact_new_variables_2}). As done previously, we get:
\begin{equation*}
\Gamma_{\Sigma} ||D^\alpha \tilde{w}(T)||_0^2 + \varepsilon^6 \Delta_\Sigma  (||D^\alpha\tilde{m}(T)||_0^2 + ||D^\alpha\tilde{\xi}(T)||_0^2) + \varepsilon^8 \Theta_\Sigma  (||D^\alpha\tilde{k}(T)||_0^2+||D^\alpha\tilde{h}(T)||_0^2) 
\end{equation*}
\begin{equation*}
+ \frac{2}{\tau} \int_0^T  \varepsilon^4 \Delta_{L\Sigma} (||D^\alpha\tilde{m}(t)||_0^2+||D^\alpha\tilde{\xi}(t)||_0^2) +  \varepsilon^6 \Theta_{L\Sigma} (||D^\alpha\tilde{k}(t)||_0^2+||D^\alpha\tilde{h}(t)||_0^2) ~ dt
\end{equation*}
\begin{equation}
\label{time_integral_1_derivative_BGK}
\le c\varepsilon^2 (||D^\alpha\textbf{u}_0||_0^2+||D^{\alpha+1} \textbf{u}_0||_0^2) + 2\int_0^T  |D^\alpha (N((\Sigma\tilde{W}(t))_1+\bar{w}), D^\alpha \tilde{W}(t))_0| ~ dt.
\end{equation}

Now, from (\ref{N_source_BGK}), 
\begin{equation*}
|(D^\alpha N(w + \bar{w}), D^\alpha \tilde{W})_0| \le \frac{1}{\tau}\{|(D^\alpha(u_1 w_2), D^\alpha
\varepsilon^2 \tilde{m}_2)_0| + |D^\alpha (u_1 w_3), D^\alpha \varepsilon^2 \tilde{m}_3)_0|
\end{equation*}
\begin{equation*}
+|(D^\alpha (u_2 w_2), D^\alpha \varepsilon^2\tilde{\xi}_2)_0|+|D^\alpha (u_2 w_3), D^\alpha \varepsilon^2 \tilde{\xi}_3)_0|\}
\end{equation*}
\begin{equation*}
\le \frac{1}{2\tau} \{||(D^\alpha(u_1 w_2)||_0^2 + ||D^\alpha (u_1 w_3)||_0^2 + ||D^\alpha (u_2 w_2)||_0^2 + ||D^\alpha (u_2 w_3)||_0^2) +  {\varepsilon^4} (||D^\alpha \tilde{m}||_0^2 + ||D^\alpha \tilde{\xi}||_0^2)\}
\end{equation*}
\begin{equation*}
\le c(||\textbf{u}||_s) ||w||_s^2 + \frac{\varepsilon^4}{2\tau} (||\tilde{m}||_s^2+||\tilde{\xi}||_s^2).
\end{equation*}

By using (\ref{w_star_w_star_tilde}) we have:
\begin{equation*}
|(D^\alpha N(w + \bar{w}), D^\alpha \tilde{W})_0| 
\quad \le {c}(||\textbf{u}||_s)(||\tilde{w}||_s^2+\varepsilon^6(||\tilde{m}||_s^2+||\tilde{\xi}||_s^2)+\varepsilon^8(||\tilde{k}||_s^2+||\tilde{h}||_s^2))  
\end{equation*}
\begin{equation*}
+ \frac{\varepsilon^4}{2\tau} (||\tilde{m}||_s^2+||\tilde{\xi}||_s^2).
\end{equation*}

Thus, from (\ref{time_integral_1_derivative_BGK}), 
\begin{equation*}
\Gamma_\Sigma || \tilde{w}(T)||_s^2 + \varepsilon^6 \Delta_\Sigma (||\tilde{m}(T)||_s^2 + ||\tilde{\xi}(T)||_s^2) +  \varepsilon^8 \Theta_\Sigma (||\tilde{k}(T)||_s^2+||\tilde{h}(T)||_s^2) 
\end{equation*}
\begin{equation*}
+ \frac{2}{\tau}  \int_0^T  \varepsilon^4 (2\Delta_{L\Sigma}-1) (||\tilde{m}(t)||_s^2+||\tilde{\xi}(t)||_s^2) + 2\varepsilon^6 \Theta_{L\Sigma} (||\tilde{k}(t)||_s^2+||\tilde{h}(t)||_s^2) ~ dt
\end{equation*}
$$ \le c\varepsilon^2 (||\textbf{u}_0||_s^2+||\nabla \textbf{u}_0||_s^2)$$
\begin{equation*}
 + c(||\textbf{u}||_{L_t^\infty H_x^s})\int_0^T ||\tilde{w}(t)||_s^2+\varepsilon^6(||\tilde{m}(t)||_s^2+||\tilde{\xi}(t)||_s^2)+\varepsilon^8(||\tilde{k}(t)||_s^2+||\tilde{h}(t)||_s^2) ~ dt.
\end{equation*}
\end{proof}

\begin{remark}
In the case $s>3$ is not an integer, by using the pseudodifferential operator $\displaystyle \lambda^s(\xi)=(1+|\xi|^2)^{s/2}$ in the Fourier space, we get the same estimates in a standard way.
\end{remark}

Now, we need a bound in the $H^s$-norm for the original variable $w=(\rho-\bar{\rho}, \varepsilon \rho \textbf{u}),$ which is the first component of the unknown vector $W$ in (\ref{W_variables_NS}). By using estimate (\ref{time_integral_1_derivative_BGK}) and definition (\ref{change_Sigma}), we can prove the following proposition.

\begin{proposition}
\label{New_constants_BGK}
If Assumptions \ref{a_condition_BGK} and \ref{lambda_assumption_BGK} are satisfied, then the following estimate holds:
\begin{equation*}
|| {w}(t)||_s^2 + \varepsilon^6 (||\tilde{m}(t)||_s^2 + ||\tilde{\xi}(t)||_s^2)+ \varepsilon^8 (||\tilde{k}(t)||_s^2+||\tilde{h}(t)||_s^2) \le c \varepsilon^2 (||\textbf{u}_0||_s^2+||\nabla \textbf{u}_0||_s^2) e^{c(||\textbf{u}||_{L_t^\infty H_x^s})t},
\end{equation*}
and 
\begin{equation}
\label{Gronwall1_proposition_BGK}
\frac{||\rho(t)-\bar{\rho}||_s^2}{\varepsilon^2}+||\rho \textbf{u}(t)||_s^2 \le c(||\textbf{u}_0||_s^2+||\nabla \textbf{u}_0||_s^2)e^{c(||\textbf{u}||_{L_t^\infty H_x^s})t},
\end{equation}
for $t \in [0,T^\varepsilon)$.
\end{proposition}

\begin{proof}
The Gronwall inequality applied to (\ref{time_integral_1_derivative_BGK}) yields:
\begin{equation}
\label{Gronwall_BGK}
\begin{aligned}
 \Gamma_\Sigma || \tilde{w}(t)||_s^2 & + \varepsilon^6 \Delta_\Sigma (||\tilde{m}(t)||_s^2 + ||\tilde{\xi}(t)||_s^2) +  \varepsilon^8 \Theta_\Sigma (||\tilde{k}(t)||_s^2+||\tilde{h}(t)||_s^2) \\\\
& \le c\varepsilon^2 (||\textbf{u}_0||_s^2+||\nabla \textbf{u}_0 ||_s^2) e^{c(||\textbf{u}||_{L_t^\infty H_x^s})t}.
\end{aligned}
\end{equation}

Recalling (\ref{w_star_w_star_tilde}),
\begin{equation*}
\tilde{w}=w-\varepsilon^3\sigma_1 \tilde{m} - \varepsilon^3 \sigma_2 \tilde{\xi} - 2a \varepsilon^4 (\tilde{k}+\tilde{h}).
\end{equation*}
Thus,
\begin{equation}
\label{back_BGK}
\begin{aligned}
||\tilde{w}||_s^2 & = ||w||_s^2 + \varepsilon^6 (||\tilde{m}_1||_s^2 + ||\tilde{m}_2||_s^2 + ||\tilde{\xi}_1||_s^2 + ||\tilde{\xi}_3||_s^2) + 4a^2 \varepsilon^8 ||\tilde{k}+\tilde{h}||_s^2\\\\
& -2\varepsilon^3(w, \sigma_1 \tilde{m})_s - 2\varepsilon^3 (w, \sigma_2 \tilde{\xi})_s-4a\varepsilon^4(w, \tilde{k}+\tilde{h})_s+2\varepsilon^6 (\sigma_1 \tilde{m}, \sigma_2 \tilde{\xi})_s \\\\
&+4a\varepsilon^7 (\sigma_1 \tilde{m}, \tilde{k}+\tilde{h})_s+4a\varepsilon^7 (\sigma_2 \tilde{\xi}, \tilde{k}+\tilde{h})_s\\\\
&= ||w||_s^2 + \varepsilon^6 (||\tilde{m}_1||_s^2 + ||\tilde{m}_2||_s^2 + ||\tilde{\xi}_1||_s^2 + ||\tilde{\xi}_3||_s^2) + 4a^2 \varepsilon^8 ||\tilde{k}+\tilde{h}||_s^2\\\\
&+Y_1+Y_2+Y_3+Y_4+Y_5+Y_6.
\end{aligned}
\end{equation}

Now, taking two positive constants $\eta, \zeta$ and using the Cauchy inequality, from (\ref{back_BGK}) we have:
\begin{equation*}
\begin{array}{l}
\begin{split}
Y_1=-2\varepsilon^3 (w, \sigma_1\tilde{m})_s
\ge -\frac{||w_1||_s^2}{\eta}-\varepsilon^6 \eta  ||\tilde{m}_2||_s^2-\frac{||w_2||_s^2}{\eta}-\varepsilon^6 \eta ||\tilde{m}_1||_s^2;
\end{split} \\\\
Y_2=-2\varepsilon^3(w, \sigma_2\tilde{\xi})_s \ge -\frac{||w_1||_s^2}{\eta}-\varepsilon^6 \eta ||\tilde{\xi}_3||_s^2-\frac{||w_3||_s^2}{\eta}-\varepsilon^6 \eta ||\tilde{\xi}_1||_s^2; \\\\
Y_3=-4a\varepsilon^4(w, \tilde{k}+\tilde{h})_s \ge \frac{-2a}{\zeta}||w||_s^2 - 2a\zeta \varepsilon^8 ||\tilde{k}+\tilde{h}||_s^2; \\\\
Y_4=2 \varepsilon^6 (\tilde{m}_2, \tilde{\xi}_3)_s \ge - \varepsilon^6(||\tilde{m}_2^\varepsilon||_s^2+||\tilde{\xi}_3^\varepsilon||_s^2); \\\\
\begin{split}
Y_5=4a\varepsilon^7[(\tilde{m}_2, \tilde{k}_1+\tilde{h}_1)_s+(\tilde{m}_1, \tilde{k}_2+\tilde{h}_2)_s] \ge -2a \varepsilon^6 \eta ||\tilde{m}_2||_s^2 - \frac{2a \varepsilon^8}{\eta} ||\tilde{k}_1+\tilde{h}_1||_s^2\\ -2a \varepsilon^6 \eta ||\tilde{m}_1||_s^2 - \frac{2a \varepsilon^8}{\eta} ||\tilde{k}_2+\tilde{h}_2||_s^2; 
\end{split}
\\ \\
\begin{split}
Y_6=4a\varepsilon^7[(\tilde{\xi}_3, \tilde{k}_1+\tilde{h}_1)_s+(\tilde{\xi}_1, \tilde{k}_3+\tilde{h}_3)_s] \ge -2a \varepsilon^6 \eta ||\tilde{\xi}_3||_s^2 - \frac{2a\varepsilon^8}{\eta}||\tilde{k}_1+\tilde{h}_1||_s^2\\ -2a \varepsilon^6 \eta ||\tilde{\xi}_1||_s^2 
- \frac{2a\varepsilon^8}{\eta}||\tilde{k}_3+\tilde{h}_3||_s^2.
\end{split}
\end{array}
\end{equation*}
The left hand side of (\ref{Gronwall_BGK}) and the previous calculations yield the following inequality:
\begin{equation}
\label{last_w_star}
\begin{aligned}
\Gamma_\Sigma || \tilde{w}||_s^2 &+ \varepsilon^6 \Delta_\Sigma (||\tilde{m}||_s^2 + ||\tilde{\xi}||_s^2)+ \varepsilon^8 \Theta_\Sigma (||\tilde{k}||_s^2+||\tilde{h}||_s^2) 
\ge  \Gamma_\Sigma \Bigg[1-\frac{2}{\eta}-\frac{2a}{\zeta}\Bigg] ||w_1||_s^2 \\\\
&+ \Gamma_\Sigma \Bigg[1-\frac{1}{\eta}-\frac{2a}{\zeta}\Bigg](||w_2||_s^2+||w_3||_s^2) 
+  \varepsilon^8\theta_\Sigma (||\tilde{k}||_s^2+||\tilde{h}||_s^2)  \\\\ 
&+ \varepsilon^6 (||\tilde{m}_1||_s^2+||\tilde{\xi}_1||_s^2) [\Delta_\Sigma + \Gamma_\Sigma (1-\eta-2a\eta)]\\\\ 
&+ \varepsilon^6 (||\tilde{m}_2||_s^2+||\tilde{\xi}_3||_s^2) [\Delta_\Sigma + \Gamma_\Sigma (-\eta-2a\eta)] 
\\\\ 
&+ \varepsilon^6 (||\tilde{m}_3||_s^2+||\tilde{\xi}_2||_s^2) \Delta_\Sigma
+\varepsilon^8||\tilde{k}_1+\tilde{h}_1||_s^2 \Gamma_\Sigma \Bigg[4a^2-2a\zeta-\frac{4a}{\eta}\Bigg]  \\\\
&+\varepsilon^8 \Gamma_\Sigma (||\tilde{k}_2+\tilde{h}_2||_s^2+||\tilde{k}_3+\tilde{h}_3||_s^2) \Bigg[4a^2-2a\zeta-\frac{2a}{\eta}\Bigg].
\end{aligned}
\end{equation}

Fixed $\beta > 1$, the Cauchy inequality yields $||\tilde{k}+\tilde{h}||_s^2 \ge (1-\frac{1}{\beta}) ||\tilde{k}||_s^2 + (1-\beta) ||\tilde{h}||_s^2,$ then the last term of (\ref{last_w_star}) is bounded from below by the following expression:

$$\varepsilon^8(||\tilde{k}_1||_s^2+||\tilde{h}_1||_s^2)\Bigg[\Theta_\Sigma + (1-{1}/{\beta})\Gamma_\Sigma \Bigg[4a^2-2a\zeta-\frac{4a}{\eta}\Bigg]\Bigg]$$

\begin{equation}
\label{wwstar_inequality}
+\varepsilon^8 (||\tilde{k}_2||_s^2+||\tilde{k}_3||_s^2+||\tilde{h}_2||_s^2+||\tilde{h}_3||_s^2)\Bigg[\Theta_\Sigma + (1-\beta) {\Gamma_\Sigma} \Bigg[4a^2-2a\zeta-\frac{2a}{\eta}\Bigg]\Bigg].
\end{equation}

Thus, in order to get estimate (\ref{Gronwall1_proposition_BGK}), we require:
\begin{equation}
\label{wwstar_conditions}
\begin{cases}
1-\frac{2}{\eta}-\frac{4a}{\zeta} > 0; \\
\Delta_\Sigma-\eta \Gamma_\Sigma (1+2a) > 0; \\
\Theta_\Sigma + (1-1/\beta){\Gamma_\Sigma}\Bigg[4a^2-2a\zeta-\frac{4a}{\eta}\Bigg] > 0; \\
\Theta_\Sigma +(1-\beta) {\Gamma_\Sigma} \Bigg[4a^2-2a\zeta-\frac{4a}{\eta}\Bigg] > 0. \\
\end{cases}
\end{equation}
Recalling definition (\ref{functional_Sigma}), $\Delta_\Sigma=2(\lambda^2 a - \delta)$, and so the second inequality is satisfied for $\lambda$ big enough. Precisely, we take $\lambda $ as in Assumption \ref{lambda_assumption_BGK} and
\begin{equation*}
\lambda>\sqrt{\frac{\delta}{a}+\frac{\eta \Gamma_\Sigma (1+2a)}{2a}}.
\end{equation*}
Moreover, the first condition of (\ref{wwstar_conditions}) is verified if 
\begin{equation}
\label{zeta_BGK}
\boxed{\eta > \frac{2\zeta}{\zeta - 4a}, \,\,\,\,\, \zeta > 4a}.
\end{equation}
Since $\Theta_\Sigma$ and $\Gamma_\Sigma$ are positive, taking $1-\beta<0$, i.e. $\beta >1$, the last inequality is verified if $$2a\zeta + \frac{4a}{\eta} - 4a^2 >0.$$ From (\ref{zeta_BGK}), 
$$2a\zeta + \frac{4a}{\eta} - 4a^2 > 8a^2 + \frac{4a}{\eta} - 4a^2 = 4a^2 + \frac{4a}{\eta} > 0,$$ then the last inequality in (\ref{wwstar_conditions}) holds under (\ref{zeta_BGK}). Now, the third condition in (\ref{wwstar_conditions}) is satisfied if 
$$ {\zeta < \frac{\Theta_\Sigma}{2a\Gamma_\Sigma (1-1/\beta)}+2(a-{1}/{\eta})}.$$
Thus, if $\displaystyle \boxed{\eta>\frac{1}{a}},$ we can take
\begin{equation}
\label{wwstar_conditions_verified}
4a < \zeta < \frac{\Theta_\Sigma}{2a\Gamma_\Sigma (1-1/\beta)},
\end{equation}
with $\eta$ and $\zeta$ satisfying (\ref{zeta_BGK}). In particular, we show that there exists $\beta>1$ such that:
\begin{equation}
\label{beta_condition_BGK}
4a<\frac{\Theta_\Sigma}{2a\Gamma_\Sigma (1-1/\beta)}, \,\,\,\,\, \text{i.e.} \,\,\,\,\, 8a^2\Gamma_\Sigma (1-1/\beta)<\Theta_\Sigma. 
\end{equation}
From (\ref{functional_Sigma}), $\Gamma_\Sigma=1-4a\mu -\frac{2}{\delta}$ and, from Lemma \ref{Sigma_positivity_lemma}, $0<\Gamma_\Sigma < 1$. Thus, in order to verify (\ref{beta_condition_BGK}), we require:
\begin{equation*}
8a^2(1-1/\beta)<\Theta_\Sigma, 
\end{equation*}
which is automatically verified if $\displaystyle 8a^2\le\Theta_\Sigma.$ Otherwise, it yields
$\displaystyle  {\beta < \frac{8a^2}{8a^2-\Theta_\Sigma}}.$
\\ Finally, since $\beta>1$, we need
$$1< \frac{8a^2}{8a^2-\Theta_\Sigma},\,\,\,\,\, \text{i.e.} \,\,\,\,\, \Theta_\Sigma > 0,$$
which is already satisfied thanks to Lemma \ref{Sigma_positivity_lemma}.

This way, from (\ref{wwstar_inequality}), (\ref{Gronwall_BGK}) and (\ref{wwstar_conditions}), we get some positive constants $\Gamma_\Sigma^1, \Delta_\Sigma^1, \Theta_\Sigma^1$ such that
$$\Gamma_\Sigma^1 || {w}(t)||_s^2 + \varepsilon^6 \Delta_\Sigma^1 (||\tilde{m}(t)||_s^2 + ||\tilde{\xi}(t)||_s^2)+ \varepsilon^8 \Theta_\Sigma^1 (||\tilde{k}(t)||_s^2+||\tilde{h}(t)||_s^2)$$
\begin{equation}
\label{Gronwall1_BGK}
 \le c \varepsilon^2 (||\textbf{u}_0||_s^2+||\nabla \textbf{u}_0||_s^2) e^{c(||\textbf{u}||_{L_t^\infty H_x^s})t},
\end{equation}

and, in particular,
\begin{equation*}
|| {w}(t)||_s^2 \le c \varepsilon^2 (||\textbf{u}_0||_s^2+||\nabla \textbf{u}_0||_s^2) e^{c(||\textbf{u}||_{L_t^\infty H_x^s})t},
\end{equation*}
i.e.
\begin{equation}
\label{s_estimate_BGK}
\frac{||\rho(t)-\bar{\rho}||_s^2}{\varepsilon^2}+||\rho \textbf{u}(t)||_s^2 \le c(||\textbf{u}_0||_s^2+||\nabla \textbf{u}_0||_s^2)e^{c(||\textbf{u}||_{L_t^\infty H_x^s})t}.
\end{equation}
\end{proof}

Thus, we are able to prove that the time $T^\varepsilon$ of existence of the solutions to the vector BGK scheme is bounded form below by a positive time $T^\star$, which is independent of $\varepsilon$.
\begin{proposition}
There exist $\varepsilon_0$ and $T^\star$ fixed such that $T^\star < T^\varepsilon$ for all $\varepsilon \le \varepsilon_0$.  This also yields, for $\varepsilon \le \varepsilon_0$, the uniform bounds:
\begin{equation}
\label{bound1_BGK}
||\textbf{u}(t)||_s \le M, \,\,\,\,\,\,\, t \in [0,T^\star],
\end{equation}
\begin{equation}
\label{bound2_BGK}
||\rho(t)-\bar{\rho}||_s \le \varepsilon M, \,\,\,\,\, i.e. \,\,\,\,\, ||\rho(t)||_s \le \bar{\rho}|\mathbb{T}^2| + \varepsilon M, \,\,\,\,\,\,\, t \in [0,T^\star],
\end{equation}
and
\begin{equation}
\label{bound3_BGK}
||\rho\textbf{u}(t)||_s \le M(\bar{\rho}|\mathbb{T}^2| + \varepsilon M), \,\,\,\,\,\,\, t \in [0,T^\star].
\end{equation}
\end{proposition}

\begin{proof}
Let $\textbf{u}_0 \in H^{s+1}(\mathbb{T}^2)$ and, from (\ref{initial_conditions_BGK}), recall that $\rho_0=\bar{\rho}.$ Then, there exists a positive constant $M_0$ such that $||\textbf{u}_0||_{s+1}\le M_0$, and 
\begin{equation}
\label{M0_BGK}
||\rho_0 \textbf{u}_0||_{s+1} = \bar{\rho} ||\textbf{u}_0||_{s+1} \le \bar{\rho} M_0=:\tilde{M}_0.
\end{equation}
Let $M > \tilde{M}_0$ be any fixed constant, and
\begin{equation}
\label{hp_uniform_bound_BGK}
T_0^\varepsilon:=\sup \Bigg\{t \in [0,T^\varepsilon]  \, \Bigg| \, \frac{||\rho(t)-\bar{\rho}||_s^2}{\varepsilon^2} + ||\rho \textbf{u}(t)||_s^2 \le M^2, \,\,\ \forall \varepsilon \le \varepsilon_0 \,\,\, \Bigg\}.
\end{equation}
Notice that, from (\ref{hp_uniform_bound_BGK}), 
\begin{equation*}
 ||\rho - \bar{\rho}||_\infty \le c_S  ||\rho - \bar{\rho}||_s \le c_S M \varepsilon , \,\, \,\, t \in [0, T_0^\varepsilon],
\end{equation*}
where $c_S$ is the Sobolev embedding constant, i.e.
$$\bar{\rho} - c_S M \varepsilon \le \rho \le \bar{\rho}+ c_S M\varepsilon, \,\, \,\, t \in [0, T_0^\varepsilon].$$
Taking $\varepsilon_0$ such that $\bar{\rho}-c_S M\varepsilon_0 > \frac{\bar{\rho}}{2}$, i.e. $\bar{\rho} > 2c_SM\varepsilon_0$, we have
\begin{equation}
\label{rho_rho_bar_2_BGK}
\rho > \frac{\bar{\rho}}{2}, \,\,\,\,\, t \in [0,T_0^\varepsilon].
\end{equation}

Now, since $s > 3=\frac{d}{2}+2$, 
$$||\textbf{u}||_s \le ||\rho \textbf{u}||_s  ||1/\rho||_s. $$ Moreover,
$$||1/\rho||_s \le c\Bigg(\frac{|\mathbb{T}^2|}{\bar{\rho}} + \frac{||\rho||_s}{c(\bar{\rho})}\Bigg) \le c_1 + c_2 ||\rho||_s. $$

From (\ref{hp_uniform_bound_BGK}), $$||\rho||_s \le c(|\mathbb{T}^2|\bar{\rho}+M\varepsilon),$$ so
$$||1/\rho||_s \le c_1 + c_2M\varepsilon,$$
and 
$$||\textbf{u}||_s \le cM (c_1+c_2M\varepsilon).$$

From (\ref{s_estimate_BGK}),
$$\frac{||\rho(t)-\bar{\rho}||_s^2}{\varepsilon^2}+||\rho \textbf{u}(t)||_s^2 \le cM_0^2 e^{c(M(c_1+c_2M\varepsilon))t}, \,\,\,\,\, t \in [0, T_0^\varepsilon].$$

We take $T^\star \le T_0^\varepsilon$ such that
$$cM_0^2 e^{c(M(c_1+c_2M\varepsilon_0))T^\star} \le M^2,$$
i.e.

\begin{equation}
\label{T_star_BGK}
T^\star \le \frac{1}{{c}(M(c_1+c_2M\varepsilon_0))}log(M^2/(cM_0^2)) \,\,\,\,\, \forall \varepsilon \le \varepsilon_0.
\end{equation}
This way,
\begin{equation}
\label{bound12_BGK}
||\textbf{u}(t)||_s \le  cM (c_1+c_2M\varepsilon), \,\,\,\,\,\,\, t \in [0,T^\star] \,\,\,\,\, \text{and} \,\,\,\,\, ||\rho \textbf{u}||_s \le M \,\,\,\,\, \forall \varepsilon \le \varepsilon_0.
\end{equation}
\end{proof}

\subsection{Time derivative estimate}
In order to use the compactness tools, we need a uniform bound for the time derivative of the unknown vector field.
\begin{proposition}
If Assumptions \ref{a_condition_BGK} and \ref{lambda_assumption_BGK} hold, for $M_0$ in (\ref{M0_BGK}) and $M$ in (\ref{bound1_BGK}), we have:
\begin{equation}
\label{gronwall_time_derivative_BGK}
\begin{array}{c}
|| \partial_t {w}||_{s-1}^2 + \varepsilon^6 (||\partial_t\tilde{m}||_{s-1}^2 + ||\partial_t\tilde{\xi}||_{s-1}^2) + \varepsilon^8 (||\partial_t\tilde{k}||_{s-1}^2+||\partial_t\tilde{h}||_{s-1}^2) \\\\
\le  \varepsilon^2 c(||\textbf{u}_0||_{s+1}) e^{c(M)t} \le \varepsilon^2 c(M_0,M) \,\,\,\,\, \text{in} \,\,\,\,  [0,T^\star],
\end{array}
\end{equation} 
with $T^\star$ in (\ref{T_star_BGK}). This also yields the uniform bound:
\begin{equation}
\label{bound_time_derivative}
\frac{||\partial_t(\rho-\bar{\rho})||_{s-1}^2}{\varepsilon^2}+||\partial_t(\rho \textbf{u})||_{s-1}^2 \le c(||\textbf{u}_0||_{s+1}) \le M^2 \,\,\,\,\, \text{in} \,\,\,\,  [0,T^\star].
\end{equation}
\end{proposition}

\begin{proof}
Let us take the time derivative of system (\ref{BGK_compact_Sigma}). Defining $\tilde{V}=\partial_t \tilde{W}^\varepsilon$, from (\ref{source_term_NS}) we get:
\begin{equation}
\label{BGK_compact_time_derivative}
\partial_t \Sigma \tilde{V} + \tilde{\Lambda}_1 \Sigma \partial_x \tilde{V} + \tilde{\Lambda}_2 \Sigma \partial_y \tilde{V} = -L\Sigma \tilde{V} + \partial_t N((\Sigma \tilde{W})_1+\bar{w})=-L\Sigma \tilde{V} +\partial_t N(w+\bar{w}),
\end{equation}
where 
\begin{equation}
\label{time_derivative_source_term}
\partial_t N(w+\bar{w})=\frac{1}{\tau}\left(\begin{array}{c}
0 \\
\left(\begin{array}{c}
0 \\
2 u_1 \partial_t w_2 - \varepsilon u_1^2 \partial_t w_1 \\
u_2 \partial_t w_2 + u_1 \partial_t w_3 - \varepsilon u_1 u_2 \partial_t w_1 \\
\end{array}\right) \\
\left(\begin{array}{c}
0 \\
u_2 \partial_t w_2 + u_1 \partial_t w_3 - \varepsilon u_1 u_2 \partial_t w_1 \\ 
2u_2 \partial_t w_3 -\varepsilon u_2^2 \partial_t w_1 \\
\end{array}\right) \\
0 \\
0 \\
\end{array}\right).
\end{equation}
Taking the scalar product with $\tilde{V}$, we have:
\begin{equation}
\label{time_derivative_1}
\frac{1}{2}\frac{d}{dt}(\Sigma \tilde{V}, \tilde{V})_0 + (L\Sigma \tilde{V}, \tilde{V})_0 \le |(\partial_t N(w+\bar{w}), V)_0|.
\end{equation}
Here,
\begin{equation*}
\begin{array}{c}
\displaystyle  |(\partial_t N(w+\bar{w}), \tilde{V})_0| = \frac{1}{\tau}|(2 u_1 \partial_t w_2 - \varepsilon u_1^2 \partial_t w_1, \varepsilon^2 \partial_t \tilde{m}_2)_0 
\\\\ \displaystyle + (u_2 \partial_t w_2 + u_1 \partial_t w_3 - \varepsilon u_1 u_2 \partial_t w_1, \varepsilon^2 \partial_t \tilde{m}_3 + \varepsilon^2 \partial_t \tilde{\xi}_2)_0
\\\\ \displaystyle+(2u_2 \partial_t w_3 -\varepsilon u_2^2 \partial_t w_1, \varepsilon^2 \partial_t \tilde{\xi}_3)_0|
\le c(||\textbf{u}||_\infty) ||\partial_t w||_0^2 + \frac{\varepsilon^4}{2\tau} (||\partial_t \tilde{m}||_0^2+||\partial_t \tilde{\xi}||_0^2).
\end{array}
\end{equation*}

Similarly to (\ref{zero_order_estimate_BGK}), we get:
\begin{equation*}
\begin{array}{c}
\Gamma_\Sigma ||\partial_t \tilde{w}||_0^2 + \varepsilon^6 \Delta_\Sigma (||\partial_t \tilde{m}||_0^2 + |||\partial_t \tilde{\xi}||_0^2) + \varepsilon^8 \Theta_\Sigma (||\partial_t \tilde{k}||_0^2+||\partial_t \tilde{h}||_0^2) 
\\\\ \displaystyle
+ \frac{1}{\tau} \int_0^T   (2\Delta_{L\Sigma}-1) \varepsilon^4 (||\partial_t\tilde{m}||_0^2+||\partial_t\tilde{\xi}||_0^2) +  2\varepsilon^6 \Theta_{L\Sigma} (||\partial_t\tilde{k}||_0^2+||\partial_t\tilde{h}||_0^2) ~ dt \\\\
\le c\varepsilon^2 ||\partial_t w|_{t=0}||_0^2 
\displaystyle \\\\ \displaystyle
+ c(||\textbf{u}||_{L^\infty([0,T]\times\mathbb{T}^2)}) \int_0^T ||\partial_t\tilde{w}||_0^2 +\varepsilon^6(||\partial_t\tilde{m}||_0^2 + ||\partial_t\tilde{\xi}||_0^2) + \varepsilon^8 (||\partial_t\tilde{k}||_0^2+||\partial_t\tilde{h}||_0^2) ~ dt.
\end{array}
\end{equation*} 

Now, from the first equation given by (\ref{BGK_NS_new_variables_translated}), 
\begin{equation*}
\partial_t w|_{t=0} = - \partial_x m|_{t=0}-\partial_y \xi|_{t=0},
\end{equation*}
where, from (\ref{variables_BGK}), (\ref{initial_conditions_BGK}), and (\ref{Fluxes_BGK}),
\begin{equation*}
m|_{t=0}=\frac{A_1(w_0)}{\varepsilon}-2a \lambda^2  \tau \partial_x w_0=\bar{\rho}\left(\begin{array}{c}
{u_0}_1 \\
\varepsilon {u_0}_1^2-2a \lambda^2 \varepsilon \partial_x {u_0}_1 \\
\varepsilon {u_0}_1{u_0}_2-2a \lambda^2 \varepsilon \partial_x {u_0}_2 \\
\end{array}\right),
\end{equation*}
\begin{equation*}
\xi|_{t=0}=\frac{A_2(w_0)}{\varepsilon}-2a \lambda^2 \tau \partial_y w_0=\bar{\rho}\left(\begin{array}{c}
{u_0}_2 \\
\varepsilon {u_0}_1{u_0}_2-2a \lambda^2 \tau \varepsilon \partial_y {u_0}_1 \\
\varepsilon {u_0}_2^2-2a \lambda^2 \tau \varepsilon \partial_y {u_0}_2 \\
\end{array}\right).
\end{equation*}
By definition of $w$ in (\ref{w_BGK}), $\displaystyle \partial_tw|_{t=0}=(\partial_t\rho|_{t=0}, \varepsilon \partial_t(\rho \textbf{u})|_{t=0}).$ This implies that
$$\partial_t \rho|_{t=0} = - {\bar{\rho}} (\nabla \cdot \textbf{u}_0) = 0,$$ since $\textbf{u}_0$ is divergence free. This way,
\begin{equation*}
\partial_t \textbf{u}|_{t=0} = - \partial_x \left(\begin{array}{c}
 {u_0}_1^2-2a \lambda^2 \tau \partial_x {u_0}_1 \\
 {u_0}_1{u_0}_2-2a \lambda^2 \tau \partial_x {u_0}_2  \\
\end{array}\right) - \partial_y \left(\begin{array}{c}
 {u_0}_1{u_0}_2-2a \lambda^2 \tau \partial_y {u_0}_1  \\
 {u_0}_2^2-2a \lambda^2 \tau \partial_y {u_0}_2 \\
\end{array}\right).
\end{equation*}

Thus, 
\begin{equation*}
\begin{array}{c}
\Gamma_\Sigma ||\partial_t \tilde{w}||_0^2 +  \varepsilon^6 \Delta_\Sigma (||\partial_t \tilde{m}||_0^2 + |||\partial_t \tilde{\xi}||_0^2) + \varepsilon^8 \Theta_\Sigma (||\partial_t \tilde{k}||_0^2+||\partial_t \tilde{h}||_0^2) \\ \\

+ \displaystyle \frac{1}{\tau}\int_0^T (2\Delta_{L\Sigma}-1)   \varepsilon^4 (||\partial_t \tilde{m}||_0^2+||\partial_t \tilde{\xi}||_0^2) + 2 \varepsilon^6 \Theta_{L\Sigma} (||\partial_t \tilde{k}||_0^2+||\partial_t {h}||_0^2) ~ dt \\ \\

\le c\varepsilon^2(||\textbf{u}_0||_0^2 + ||\nabla \textbf{u}_0||_0^2+|| \nabla^2 \textbf{u}_0||_0^2) \\\\

+ c(M) \displaystyle{\int_0^T}  ||\partial_t \tilde{w}||_0^2 +\varepsilon^6(||\partial_t \tilde{m}||_0^2 + ||\partial_t \tilde{\xi}||_0^2) + \varepsilon^8 (||\partial_t \tilde{k}||_0^2+||\partial_t \tilde{h}||_0^2) ~ dt,
\end{array}
\end{equation*} 
where the last inequality follows form the Sobolev embedding theorem and from (\ref{bound1_BGK}).

Similarly, taking the $|\alpha|$-derivative, for $|\alpha| \le s-1,$ of (\ref{BGK_compact_time_derivative}) and multiplying by $D^\alpha \tilde{V}$ through the scalar product, we get:
\begin{equation*}
\frac{1}{2}\frac{d}{dt}(\Sigma D^\alpha \tilde{V}, D^\alpha \tilde{V})_0 + (L\Sigma D^\alpha \tilde{V}, D^\alpha \tilde{V})_0 \le |(D^\alpha \partial_t N(w+\bar{w}), D^\alpha V)_0|,
\end{equation*}
where
\begin{equation*}
\begin{array}{c}
\displaystyle  |(D^\alpha \partial_t N(w+\bar{w}), D^\alpha \tilde{V})_0| = \frac{1}{\tau}|(D^\alpha (2 u_1 \partial_t w_2 - \varepsilon u_1^2 \partial_t w_1), \varepsilon^2 \partial_t D^\alpha \tilde{m}_2)_0
\\\\ \displaystyle + (D^\alpha (u_2 \partial_t w_2 + u_1 \partial_t w_3 - \varepsilon u_1 u_2 \partial_t w_1), \varepsilon^2 \partial_t D^\alpha \tilde{m}_3 + \varepsilon^2 \partial_t D^\alpha \tilde{\xi}_2)_0
\\\\  \displaystyle+(D^\alpha (2u_2 \partial_t w_3 -\varepsilon u_2^2 \partial_t w_1), \varepsilon^2 \partial_t D^\alpha \tilde{\xi}_3)_0|
\\\\  \displaystyle\le c(||\textbf{u}||_{s-1}) ||\partial_t w||_{s-1}^2 + \frac{\varepsilon^4}{2\tau} (||\partial_t \tilde{m}||_{s-1}^2+||\partial_t \tilde{\xi}||_{s-1}^2) 
\\\\ \displaystyle \le c(M) ||\partial_t w||_{s-1}^2 + \frac{\varepsilon^4}{2\tau} (||\partial_t \tilde{m}||_{s-1}^2+||\partial_t \tilde{\xi}||_{s-1}^2),
\end{array}
\end{equation*}

where the last inequality follows from (\ref{bound1_BGK}). Finally, we obtain:
\begin{equation}
\label{higher_order_estimate_time_derivative1}
\begin{array}{c}
\Gamma_\Sigma ||\partial_t \tilde{w}||_{s-1}^2 +  \varepsilon^6  \Delta_\Sigma (||\partial_t \tilde{m}||_{s-1}^2 + ||\partial_t \tilde{\xi}||_{s-1}^2) + \varepsilon^8 \Theta_\Sigma (|| \partial_t \tilde{k}||_{s-1}^2+||\partial_t \tilde{h}||_{s-1}^2) \\\\

+ \displaystyle  \frac{1}{\tau} {\int_0^T}   (2\Delta_{L\Sigma}-1) \varepsilon^4 (||\partial_t \tilde{m}||_{s-1}^2+||\partial_t \tilde{\xi}||_{s-1}^2) + \varepsilon^6 \Theta_{L\Sigma} (||\partial_t \tilde{k}||_{s-1}^2+||\partial_t \tilde{h}||_{s-1}^2) ~ dt \\\\

\le c\varepsilon^2 (||\textbf{u}_0||_{s-1}^2 + ||\nabla \textbf{u}_0||_{s-1}^2 + ||\nabla^2 \textbf{u}_0||_{s-1}^2 ) 

\\\\ + c(M) \displaystyle{\int_0^T}  ||\partial_t\tilde{w}||_{s-1}^2 +\varepsilon^6(||\partial_t\tilde{m}||_{s-1}^2 + ||\partial_t\tilde{\xi}||_{s-1}^2) + \varepsilon^8 (||\partial_t\tilde{k}||_{s-1}^2+||\partial_t\tilde{h}||_{s-1}^2) ~ dt.
\end{array}
\end{equation} 

\begin{lemma}
\label{New_constants_time_derivative}
If Assumption \ref{a_condition_BGK} and \ref{lambda_assumption_BGK} hold, then there exists a positive constant $c$ such that:
\begin{equation*}
\begin{array}{c}
||\partial_t {w}||_{s-1}^2 
\le c(||\partial_t \tilde{w}||_{s-1}^2 +  \varepsilon^6  (||\partial_t \tilde{m}||_{s-1}^2 + ||\partial_t \tilde{\xi}||_{s-1}^2)
+ \varepsilon^8 (|| \partial_t \tilde{k}||_{s-1}^2+||\partial_t \tilde{h}||_{s-1}^2)).
\end{array}
\end{equation*}
\end{lemma}
\begin{proof}
The proof of Proposition \ref{New_constants_BGK} can be adapted here with slight modifications.
\end{proof}
We end the proof by applying the Gronwall inequality to (\ref{higher_order_estimate_time_derivative1}) and using Lemma \ref{New_constants_time_derivative}.
\end{proof}

\section{Convergence to the Navier-Stokes equations}
\label{Convergence_NS}
Now we state our main result.
\begin{theorem}
\label{Main_Theorem}
Let $s>3$. If Assumptions \ref{a_condition_BGK} and \ref{lambda_assumption_BGK} hold, there exists a subsequence $W^\varepsilon=(w^{\varepsilon \, \star}, \varepsilon^2 m^\varepsilon, \varepsilon^2 \xi^\varepsilon, \varepsilon^2 k^{\varepsilon \, \star}, \varepsilon^2 h^{\varepsilon \, \star}),$ with $w^{\varepsilon \, \star}=(\rho^\varepsilon-\bar{\rho}, \varepsilon \rho^\varepsilon \textbf{u}^\varepsilon)$ and $\bar{\rho}>0$, of the solutions to the vector BGK model (\ref{BGK_NS_compact_new_variables}) with  initial data (\ref{initial_conditions_compact_new_variables_NS}) and $\textbf{u}_0 \in H^{s+1}(\mathbb{T}^2)$ in (\ref{real_NS_initial_data}), such that 
$$(\rho^\varepsilon, \textbf{u}^\varepsilon) \rightarrow  (\bar{\rho}, \textbf{u}^{NS}) \,\,\,\, \text{in} \,\,\,\, C([0,T^\star], {H^s}'(\mathbb{T}^2)),$$ with $T^\star$ in (\ref{T_star_BGK}), $ s-1 < s' < s $,  and where $\textbf{u}^{NS}$ is the unique solution to the Navier-Stokes equations in (\ref{real_NS}),
with initial data $\textbf{u}_0$ above and $P^{NS}$ the incompressible pressure. Moreover, 
$$\frac{\nabla (\rho^\varepsilon-\bar{\rho})}{\varepsilon^2} \rightharpoonup^\star \nabla P^{NS} \,\,\,\,\, \text{in} \,\,\,\,\, {L_t^\infty H_x^{s-3}}.$$
\end{theorem}

\begin{proof}
First of all, consider the previous bounds in (\ref{bound1_BGK}), (\ref{bound2_BGK}), (\ref{bound3_BGK}) and (\ref{bound_time_derivative}):

\begin{equation}
\label{bounds_rho}
\sup_{t \in [0, T^\star]} \frac{||\rho^\varepsilon-\bar{\rho}||_s}{\varepsilon} \le M, \,\,\,\,\, \sup_{t \in [0, T^\star]} \frac{||\partial_{t}(\rho^\varepsilon-\bar{\rho})||_{s-1}}{\varepsilon} \le M_1,
\end{equation}

\begin{equation}
\label{bounds_u}
\sup_{t \in [0, T^\star]} ||\rho^\varepsilon \textbf{u}^\varepsilon||_s \le N, \,\,\,\,\, \sup_{t \in [0, T^\star]} ||\partial_t(\rho^\varepsilon \textbf{u}^\varepsilon)||_{s-1} \le N_1,
\end{equation}
where $M, M_1, N, N_1$ are positive constants. The Lions-Aubin Lemma in \cite{Bertozzi} implies that, for $s-1<s'<s$,
$$\rho^\varepsilon \rightarrow \bar{\rho} \,\,\,\, \text{strongly} \,\,\,\, \text{in} \,\,\,\, C([0, T^\star], {H^s}'(\mathbb{T}^2)), \,\,\, $$
and there exists $\textbf{m}^\star$ such that
$$\textbf{m}^\varepsilon=\rho^\varepsilon \textbf{u}^\varepsilon \rightarrow \textbf{m}^\star  \,\,\,\, \text{strongly} \,\,\,\, \text{in} \,\,\,\, C([0, T^\star], {H^s}'(\mathbb{T}^2)).$$
Notice also that $\displaystyle \textbf{u}^\varepsilon=\frac{\textbf{m}^\varepsilon}{\rho^\varepsilon}$, where $$\displaystyle 1/\rho^\varepsilon \rightarrow  1/\bar{\rho}  \,\,\,\, \text{strongly} \,\,\,\, \text{in} \,\,\,\, C([0, T^\star], {H^s}'(\mathbb{T}^2)),$$
since we can take $\bar{\rho}$ such that $\rho^\varepsilon > \frac{\bar{\rho}}{2}$ as in (\ref{rho_rho_bar_2_BGK}). Then
$$\displaystyle \textbf{u}^\varepsilon=\frac{\textbf{m}^\varepsilon}{\rho^\varepsilon} \rightarrow \frac{\textbf{m}^\star}{\bar{\rho}}=:\textbf{u}^\star\,\,\,\, \text{strongly} \,\,\,\, \text{in} \,\,\,\, C([0, T^\star], {H^s}'(\mathbb{T}^2)).$$ 

Now, consider system (\ref{BGK_NS_new_variables_translated}) in the following formulation:
\begin{equation}
\label{BGK_un_eps}
\begin{cases}
& \partial_{t}w^{\varepsilon}+\partial_{x}m^\varepsilon+\partial_{y}\xi^\varepsilon=0; \\
& \varepsilon \partial_{t} m^\varepsilon + \frac{\lambda^{2}}{\varepsilon}\partial_{x}k^{\varepsilon}=\frac{1}{\tau}(\frac{A_{1}(w^{\varepsilon}+\bar{w})}{\varepsilon^2}-\frac{m^\varepsilon}{\varepsilon}), \\
& \varepsilon \partial_{t}\xi^\varepsilon+\frac{{\lambda^{2}}}{\varepsilon}\partial_{y}h^{\varepsilon}=\frac{1}{\tau}(\frac{A_{2}(w^{\varepsilon}+\bar{w})}{\varepsilon^2}-\frac{\xi^\varepsilon}{\varepsilon}), \\
& \varepsilon \partial_{t}k^{\varepsilon}+\varepsilon \partial_{x}m^\varepsilon=\frac{(2aw^{\varepsilon}-k^{\varepsilon})}{\tau \varepsilon}, \\
& \varepsilon \partial_{t}h^{\varepsilon}+\varepsilon \partial_{y}\xi^\varepsilon=\frac{(2aw^{\varepsilon}-h^{\varepsilon})}{\tau \varepsilon}, \\
\end{cases}.
\end{equation}

From (\ref{BGK_un_eps}) and $2a\lambda^2\tau={\nu}$ as in (\ref{a_def_BGK}), it follows that
\begin{equation*}
\begin{cases}
m^\varepsilon=\frac{A_1(w^\varepsilon+\bar{w})}{\varepsilon}-{\nu}\partial_x w^\varepsilon+\varepsilon^2\lambda^2\tau^2(\partial_{tx}k^\varepsilon+\partial_{xx}m^\varepsilon)-\varepsilon^2\tau\partial_t m^\varepsilon; \\
\xi=\frac{A_2(w^\varepsilon+\bar{w})}{\varepsilon}-{\nu}\partial_y w^\varepsilon + \varepsilon^2 \lambda^2\tau^2(\partial_{ty}h^\varepsilon+\partial_{yy} \xi^\varepsilon) - \varepsilon^2\tau \partial_t \xi^\varepsilon.\\
\end{cases}
\end{equation*}

Substituting the expansions above in the first equation of (\ref{BGK_un_eps}), we get the following equation:
\begin{equation*}
\begin{array}{c}
\partial_t w^\varepsilon +  \dfrac{\partial_x A_1(w^\varepsilon+\bar{w})}{\varepsilon}+ \dfrac{\partial_y A_2(w^\varepsilon+\bar{w})}{\varepsilon}-{\nu}\Delta w^\varepsilon \\\\

=\varepsilon^2\tau \partial_{tx}m^\varepsilon+\varepsilon^2\tau \partial_{ty}\xi^\varepsilon - \varepsilon^2 \lambda^2\tau^2 (\partial_{txx}k^\varepsilon+\partial_{xxx}m^\varepsilon+\partial_{tyy}h^\varepsilon+\partial_{yyy}\xi^\varepsilon).
\end{array}
\end{equation*}

We recall that $W^\varepsilon=\Sigma \tilde{W}^\varepsilon$ by definition (\ref{change_Sigma}), with $W^\varepsilon, \tilde{W}^\varepsilon$ in (\ref{W_variables_NS}) and (\ref{change_Sigma}) respectively. This yields:
\begin{equation}
\label{sigma}
\begin{cases}
w^\varepsilon=\tilde{w}^\varepsilon+\varepsilon^3\sigma_1 \tilde{m}^\varepsilon+\varepsilon^3 \sigma_2 \tilde{\xi}^\varepsilon + 2a\varepsilon^4 \tilde{k}^\varepsilon+2a\varepsilon^4\tilde{h}^\varepsilon; \\

\varepsilon^2 m^\varepsilon=\varepsilon \sigma_1 \tilde{w}^\varepsilon + 2 a \lambda^2 \varepsilon^4 \tilde{m}^\varepsilon+ \varepsilon^5 \sigma_1 \tilde{k}^\varepsilon; \\

\varepsilon^2 \xi^\varepsilon=\varepsilon \sigma_2 \tilde{w}^\varepsilon + 2 a \lambda^2 \varepsilon^4 \tilde{\xi}^\varepsilon + \varepsilon^5 \sigma_2 \tilde{h}^\varepsilon; \\

\varepsilon^2 k^\varepsilon=2 a \varepsilon^2 \tilde{w}^\varepsilon + \varepsilon^5 \sigma_1 \tilde{m}^\varepsilon + 2 a \varepsilon^6 \tilde{k}^\varepsilon; \\

\varepsilon^2 h^\varepsilon=2 a \varepsilon^2 \tilde{w}^\varepsilon + \varepsilon^5 \sigma_2 \tilde{\xi}^\varepsilon + 2 a \varepsilon^6 \tilde{h}^\varepsilon.
\end{cases}
\end{equation}

From (\ref{gronwall_time_derivative_BGK}), (\ref{Gronwall1_proposition_BGK})-(\ref{bound12_BGK}) and (\ref{sigma}) it follows that, for a fixed constant value $c>0$, 
\begin{equation*}
\tau\varepsilon^2 ||\partial_{tx}m^\varepsilon+\partial_{ty}\xi^\varepsilon - \lambda^2 \tau (\partial_{txx}k^\varepsilon+\partial_{xxx}m^\varepsilon+\partial_{tyy}h^\varepsilon + \partial_{yyy}\xi^\varepsilon)||_{s-3} = O(\varepsilon^2),
\end{equation*}

then

\begin{equation}
\label{convergence_BGK}
\Bigg\|\partial_t w^\varepsilon +  \frac{\partial_x A_1(w^\varepsilon+\bar{w})}{\varepsilon}+ \frac{\partial_y A_2(w^\varepsilon+\bar{w})}{\varepsilon}-{\nu}\Delta w^\varepsilon\Bigg\|_{s-3} = O(\varepsilon^2).
\end{equation}

The last two equations and the previous bounds (\ref{bounds_rho}) and (\ref{bounds_u}) yield:
\begin{equation}
\label{last_bound_BGK}
\Bigg\|\partial_{t}(\rho^\varepsilon \textbf{u}^\varepsilon)+\nabla \cdot (\rho^\varepsilon \textbf{u}^\varepsilon \otimes \textbf{u}^\varepsilon) +  \frac{\nabla(\rho^\varepsilon-\bar{\rho})}{\varepsilon^2} -{\nu}\Delta (\rho \textbf{u}^\varepsilon) \Bigg\|_{s-3}=O(\varepsilon),
\end{equation}

and, in particular, 
\begin{equation*}
\frac{||\nabla(\rho^\varepsilon-\bar{\rho})||_{s-3}}{\varepsilon^2} \le c, 
\end{equation*}
i.e. there exists $\nabla P^\star \in {L_t^\infty H_x^{s-3}}$ such that
\begin{equation*}
\frac{\nabla(\rho^\varepsilon - \bar{\rho})}{\varepsilon^2} \rightharpoonup^\star \nabla P^\star \,\,\,\,\, \text{in} \,\,\,\,\,{L_t^\infty H_x^{s-3}}.
\end{equation*}
Moreover, since $\rho^\varepsilon \rightarrow \bar{\rho} \,\,\, \text{and} \,\,\,\, \textbf{u}^\varepsilon \rightarrow \textbf{u}^\star \,\,\,\,  \text{in} \,\,\,\, C([0, T^\star], {H^s}'(\mathbb{T}^2))$, from $||\partial_t (\rho^\varepsilon \textbf{u}^\varepsilon)||_{s-1} \le N_1$ as in (\ref{bounds_u}), it follows also that 
\begin{equation*}
\partial_t(\rho^\varepsilon \textbf{u}^\varepsilon) \rightharpoonup^\star \bar{\rho}\partial_t \textbf{u}^\star \,\,\,\, \text{in} \,\,\,\, {L_t^\infty H_x^{s-3}},
\end{equation*} 
while
\begin{equation*}
\nabla \cdot (\rho^\varepsilon \textbf{u}^\varepsilon \otimes \textbf{u}^\varepsilon) \rightharpoonup^\star \bar{\rho} \nabla \cdot (\textbf{u}^\star \otimes \textbf{u}^\star) \,\,\,\, \text{in} \,\,\,\, {L_t^\infty H_x^{s-3}}.
\end{equation*}

Thus, from (\ref{last_bound_BGK}) we have the weak$^\star$ convergence in ${L_t^\infty H_x^{s-3}},$ i.e.
\begin{equation*}
\partial_t(\rho^\varepsilon \textbf{u}^\varepsilon)+\nabla \cdot (\rho^\varepsilon \textbf{u}^\varepsilon \otimes \textbf{u}^\varepsilon) + \frac{\nabla (\rho^\varepsilon-\bar{\rho})}{\varepsilon^2} - {\nu} \Delta (\rho^\varepsilon \textbf{u}^\varepsilon) \rightharpoonup^\star \bar{\rho} \Bigg(\partial_t \textbf{u}^\star + \nabla \cdot (\textbf{u}^\star \otimes \textbf{u}^\star) + \frac{\nabla P^\star}{\bar{\rho}} - {\nu}\Delta \textbf{u}^\star \Bigg).
\end{equation*}

On the other hand, the first equation of (\ref{convergence_BGK}) yields
\begin{equation}
\label{null_divergence_BGK}
\partial_t (\rho^\varepsilon-\bar{\rho}) + \nabla \cdot (\rho^\varepsilon \textbf{u}^\varepsilon)-{\nu} \Delta (\rho^\varepsilon -\bar{\rho})=O(\varepsilon^2).
\end{equation}

Notice that $||\partial_t(\rho^\varepsilon-\bar{\rho})||_{s-1} = O(\varepsilon)$ and $||\Delta (\rho^\varepsilon -\bar{\rho})||_{s-2} = O(\varepsilon)$ thanks to (\ref{bounds_rho}), while $$\rho^\varepsilon \rightarrow \bar{\rho} \quad \text{and} \quad \textbf{u}^\varepsilon \rightarrow \textbf{u}^\star \quad \text{in} \quad C([0, T^\star], {H^s}'(\mathbb{T}^2)).$$ This way, from (\ref{null_divergence_BGK}) we finally recover the divergence free condition
\begin{equation*}
\nabla \cdot \textbf{u}^\star=0.
\end{equation*}
\end{proof}

\section{Conclusions and perspectives}
\label{Conclusions}
In this paper we proved the convergence of the solutions to the vector BGK model to the solutions to the incompressible Navier-Stokes equations on the two dimensional torus $\mathbb{T}^2.$ It could be worth extending these results to the whole space and to a general bounded domain with suitable boundary conditions, but new ideas are needed to approach these cases. Rather than the more classical kinetic entropy approach, in this paper our main tool was the use of a constant right symmetrizer, which provides the conservative-dissipative form introduced in \cite{Bianchini}, and allows us to get higher order energy estimates. Another interesting problem is to estimate the rate of convergence, in terms of the difference $||\textbf{u}^\varepsilon-\textbf{u}^{NS}||_s,$ with $\textbf{u}^\varepsilon, \textbf{u}^{NS}$ the velocity fields associated with the BGK system in (\ref{q_BGK}) and the Navier-Stokes equations in (\ref{real_NS}) respectively.

\end{document}